\documentclass[12pt]{amsart}

\usepackage{amsfonts}
\usepackage{amsmath}
\usepackage{amssymb}

\usepackage{amscd}
\usepackage[all]{xy}
\usepackage[pdftex,final]{graphicx}

\newtheorem{lem}{Lemma}[section]
\newtheorem{thm}[lem]{Theorem}
\newtheorem{prop}[lem]{Proposition}
\newtheorem{cor}[lem]{Corollary}
\theoremstyle{definition}

\newtheorem{exa}[lem]{Example}
\newtheorem{rem}[lem]{Remark}

\newcommand{\Q}{\Bbb{Q}}
\newcommand{\F}[1]{\Bbb{F}_{#1}}

\newcommand{\Z}{\Bbb{Z}}

\newcommand{\C}{\Bbb{C}}

\newcommand{\oddpart}[1]{\left( #1\right)_{\mbox{ \tiny odd}}}

\newcommand{\spl}[2]{\operatorname{SL}_{#1}\left(#2\right)}

\newcommand{\gl}[2]{\mathrm{GL}_{#1}(#2)}

\newcommand{\pb}[1]{\mathcal{P}(#1)}
\newcommand{\qpb}[1]{\widetilde{\mathcal{P}}(#1)}
\newcommand{\rpb}[1]{{\mathcal{RP}}(#1)}
\newcommand{\rpbker}[1]{\mathcal{RP}_1(#1)}
\newcommand{\qrpb}[1]{\widetilde{\mathcal{RP}}(#1)}
\newcommand{\rpbp}[1]{\mathcal{RP}_+(#1)} 
\newcommand{\qrpbker}[1]{\widetilde{\mathcal{RP}}_1(#1)}
\newcommand{\bl}[1]{\mathcal{B}(#1)}

\newcommand{\rbl}[1]{{\mathcal{RB}}(#1)}
\newcommand{\gpb}[1]{\left[ #1\right]} 
\newcommand{\sgpb}[1]{[ #1 ]}
\newcommand{\sus}[1]{\left\{ #1\right\} }
\newcommand{\suss}[2]{\psi_{#1}\left( #2\right)}

\newcommand{\rfmod}[2]{{#1}\left\{ #2\right\}}

\newcommand{\kv}[1]{\mathcal{L}_{#1}}

\newcommand{\bconst}[1]{C_{#1}}

\newcommand{\cconstmod}[1]{\mathcal{D}_{#1}}

\newcommand{\bw}{\Lambda}

\newcommand{\spec}[1]{S_{#1}}
\newcommand{\sspec}[1]{\bar{S}_{#1}}
\newcommand{\pn}[2]{\mathbb{P}^{#1}(#2)}
\newcommand{\projl}[1]{\pn{1}{#1}}

\newcommand{\sq}[1]{{#1}^\times/({#1}^\times)^2}

\newcommand{\primes}{\mathrm{Primes}}

\newcommand{\laur}[2]{{#1}[{#2},{#2}^{-1}]}

\newcommand{\card}[1]{\left| #1 \right|}

\renewcommand{\forall}{\mbox{ for all }}

\newcommand{\dual}[1]{\widehat{#1}}

\renewcommand{\ker}[1]{\mathrm{Ker}(#1)}

\newcommand{\hmz}[2]{\mathrm{Hom}(#1,#2)}

\newcommand{\unitr}[1]{{#1}^\times}

\newcommand{\modtwo}[1]{{#1}/2}

\newcommand{\ptor}[2]{{#1}_{(#2)}}
\newcommand{\ntor}[2]{_{#1}{#2}}

\newcommand{\sgr}[1]{\mathrm{R}_{#1}}
\newcommand{\sgrp}[1]{\mathrm{R}^+_{#1}}
\newcommand{\an}[1]{\left\langle{#1}\right\rangle}
\newcommand{\pf}[1]{\left\langle\!\left\langle{#1}\right\rangle\!\right\rangle}
\newcommand{\aug}[1]{\mathcal{I}_{#1}}

\newcommand{\pows}[2]{#1\left[\!\left[ #2 \right]\!\right]}
\newcommand{\laurs}[2]{{#1}\left(\!\left( #2 \right)\!\right)}

\newcommand{\ep}[1]{\mathrm{e}^{#1}_+} 
\newcommand{\epm}[1]{\mathrm{e}^{#1}_-} 

\newcommand{\upp}[2]{{#1}^{#2}}
\newcommand{\dwn}[2]{{#1}_{#2}}

\newcommand{\asymm}{\circ}
\newcommand{\asym}[3]{\mathrm{S}^{#1}_{#2}(#3)}

\newcommand{\zhalf}[1]{{#1}\left[ \tfrac{1}{2}\right]}

\newcommand{\znth}[2]{{#1}\left[\frac{1}{#2}\right]}

\newcommand{\kind}[1]{K^{\mathrm{\small ind}}_3(#1)}

\newcommand{\ho}[3]{\mathrm{H}_{#1}\left( #2,#3 \right)}

\newcommand{\hoz}[2]{\ho{#1}{#2}{\Z}}
\newcommand{\hot}[2]{\mathrm{H}_3\left( \spl{2}{#1},#2\right)_0}

\newcommand{\ind}[2]{\mathrm{Ind}^{#1}_{#2}}

\newcommand{\Tor}[2]{\mathrm{Tor}^{\Z}_{1}(#1,#2)}

\title{The third homology  of $\mathrm{SL}_2(\Q)$ 
}
\author{Kevin Hutchinson}
\address{School of Mathematics and Statistics,
 University College Dublin, Belfield, Dublin 4, Ireland}
\email{kevin.hutchinson@ucd.ie}
\date{\today}


\begin{document}
\bibliographystyle{plain}
\maketitle

\begin{abstract}
We calculate the structure of 
$\ho{3}{\spl{2}{\Q}}{\zhalf{\Z}}$. Let $\hot{\Q}{\Z}$ denote the kernel of the (split) surjective homomorphism
$\ho{3}{\spl{2}{\Q}}{\Z}\to \kind{\Q}$. Each prime number $p$ determines an operator $\an{p}$ on $\ho{3}{\spl{2}{\Q}}{\Z}$ with square the identity. 
We prove that $\hot{\Q}{\zhalf{\Z}}$ is the direct sum of the $(-1)$-eigenspaces of these operators. The $(-1)$-eigenspace of $\an{p}$ is the scissors congruence group, over 
$\zhalf{\Z}$, of the field $\F{p}$, which is a cyclic group whose order is
the odd part of ${p+1}$.
\end{abstract}

\section{Introduction}\label{sec:intro}
Many years ago, in an article on the homology of Lie groups made discrete, Chi-Han Sah, quoting S. Lichtenbaum, cited our lack of any precise knowledge of the structure of 
$\ho{3}{\spl{2}{\Q}}{\Z}$ as an example of the poor state of our understanding of the homology of linear groups of general fields
(see \cite[pp 307-8]{sah:discrete3}). Where such understanding does exist, even now, it tends often to come from connections with algebraic $K$-theory or Lie group theory where a bigger suite of mathematical tools is available. For example, we know the structure of $\ho{3}{\spl{3}{\Q}}{\Z}$ because homology stability theorems tell us that it is isomorphic to $\ho{3}{\spl{n}{\Q}}{\Z}$ for all larger $n$ (\cite{hutchinson:tao2}) and this stable homology group is in turn isomorphic, via a Hurewicz homomorphism,  to $K_3(\Q)/\{ -1\}\cdot K_2(\Q)=\kind{\Q}$ (indecomposable $K_3$) by \cite[Lemma 5.2]{sus:bloch}, which is known to be  cyclic of order $24$ by the result of Lee and Szczarba (\cite{lee:szczarba}). 

For any field $F$, the natural map $\ho{3}{\spl{2}{F}}{\Z}\to \ho{3}{\spl{3}{F}}{\Z} \cong  K_3(F)/\{ -1\}\cdot K_2(F)\to \kind{F}$ can be shown to be surjective (\cite{hutchinson:tao2}). When $F=\C$, or more generally when $F$ is algebraically closed, it has long been known, thanks to the work of Sah and his co-authors,  that this map is an isomorphism.
 Note that when $F$ is a number field, or a global function field, the map $\ho{3}{\spl{3}{F}}{\Z}\to \kind{F}$ is 
an isomorphism, since $\ho{3}{\spl{2}{F}}{\Z}\cong \ho{3}{\spl{\infty}{F}}{\Z}$ by stability (see \cite{hutchinson:tao2}), $\ho{3}{\spl{\infty}{F}}{\Z}\cong K_3(F)/\{ -1\}\cdot K_2(F)$ by
\cite[Lemma 5.2]{sus:bloch} and furthermore $\{ -1\}\cdot K_2(F)=K_3^M(F)$ (Milnor $K_3$) by the calculations of \cite{bass:tate}.  Thus, for any number field $F$, the kernel of the map $\ho{3}{\spl{2}{F}}{\Z}\to \kind{F}$ is 
 just the kernel of the stability homomorphism  $\ho{3}{\spl{2}{F}}{\Z}\to \ho{3}{\spl{3}{F}}{\Z}$. 

One natural obstruction to the injectivity or surjectivity of the stability homomorphisms\\
 $\ho{\bullet}{\spl{n}{F}}{\Z}\to \ho{\bullet}{\spl{n+1}{F}}{\Z}$ lies in the action of the multiplicative group 
$F^\times$: For any $a\in F^\times$ conjugation on $\spl{n}{F}$ by a matrix $M$ of determinant $a$ induces an automorphism of $\ho{\bullet}{\spl{n}{F}}{\Z}$ which depends only on $a$. In particular, 
$a^n=\det(\mathrm{diag}(a,\ldots,a))$ acts trivially. Since the stability homomorphism is a map of $\Z[F^\times]$ modules, both $a^n$ and $a^{n+1}$ act trivially on its image, and so the action of 
$F^\times$ on the image of this map is  trivial. It follows that the stability homomorphism factors through the coinvariants of $F^\times$ on $ \ho{\bullet}{\spl{n}{F}}{\Z}$ and has image lying in the  
invariants of $F^\times$ on $ \ho{\bullet}{\spl{n+1}{F}}{\Z}$. In particular, when $F^\times$ acts nontrivially on $ \ho{\bullet}{\spl{n}{F}}{\Z}$, the stability homomorphism has a nontrivial kernel,
since it contains $\aug{F}\ho{\bullet}{\spl{n}{F}}{\Z}$, where $\aug{F}$ denotes the augmentation ideal of the group ring $\Z[F^\times]$. 

For example, the calculations of Suslin in \cite{sus:tors} tell us that for any infinite (or sufficiently large)  field $F$ the map $\ho{2}{\spl{2}{F}}{\Z}\to \ho{2}{\spl{3}{F}}{\Z}\cong K_2(F)$ is surjective
with kernel $\aug{F}\ho{2}{\spl{2}{F}}{\Z}$ isomorphic to $I(F)^3$ where $I(F)$ denotes the fundamental ideal in the Grothendieck-Witt ring of the field $F$. In the case $F=\Q$, this kernel is isomorphic 
to the $\Z[\Q^\times]$-module $\Z$ on which $-1$ acts by negation and all primes act trivially.  

  B. Mirzaii has shown (\cite{mirzaii:third}) for infinite fields $F$ that the kernel of the stability homomorphism $\ho{3}{\spl{2}{F}}{\Z}\to \ho{3}{\spl{3}{F}}{\Z}=\kind{F}$, when tensored with $\zhalf{\Z}$,
 is  $\aug{F}\ho{3}{\spl{2}{F}}{\zhalf{\Z}}$; i.e., it is again the case that the only obstruction to injective stability is the nontriviality of the action of the multiplicative group. He subsequently (\cite{mirzaii:manyunits}) generalised this result to rings with many units (including local rings with infinite residue fields).

The main theorem of this article (Theorem \ref{thm:main}) describes the structure of $\aug{\Q}\ho{3}{\spl{2}{\Q}}{\zhalf{\Z}}$ as a $\Z[\Q^\times]$-module. $-1\in \Q^\times$ acts trivially, but each prime acts nontrivially. Since the 
the squares of rational numbers act trivially, each prime induces a decomposition into $(+1)$- and $(-1)$-eigenspace.  The theorem states that this module is the direct sum over all primes of 
these $(-1)$-eigenspaces. The $(-1)$-eigenspace of the prime $p$ is isomorphic, via a natural residue homomorphism $S_p$, to $\zhalf{\pb{\F{p}}}$, the \emph{scissors congruence  group} of the field $\F{p}$. It follows that as an abelian group 
\[
\ho{3}{\spl{2}{\Q}}{\zhalf{\Z}}\cong \zhalf{\kind{\Q}}\oplus\left( \bigoplus_{p}\zhalf{\pb{\F{p}}}\right)\cong \Z/3\oplus \left(\bigoplus_p \Z/\oddpart{p+1}\right)
\]  
where $\oddpart{m}$ denotes the odd part of $m\in \Q^\times$; i.e. $\oddpart{m}=2^{-v_2(m)}m$. 

As explained in Section \ref{sec:main} below, this theorem can be stated equivalently as follows: For any field $F$, let $\hot{F}{\Z}$ denote the kernel of the 
surjective homomorphism $\ho{3}{\spl{2}{F}}{\Z}\to \kind{F}$.  Then the map $\ho{3}{\spl{2}{\Q}}{\Z}\to \prod_p \ho{3}{\spl{2}{\Q_p}}{\Z}$ (product over all primes)
induces an isomorphism
\[
\hot{\Q}{\zhalf{\Z}}\cong \bigoplus_p \hot{\Q_p}{\zhalf{\Z}}.
\]
(In Section \ref{sec:val} we give a new more streamlined proof of the identification 
\[
\hot{\Q_p}{\zhalf{\Z}}\cong \zhalf{\pb{\F{p}}}.)
\]

The main tool we use is the description of $\ho{3}{\spl{2}{F}}{\zhalf{\Z}}$ in terms of \emph{refined} scissors congruence groups. The scissors congruence group $\pb{F}$ of a field $F$ was introduced by Dupont and Sah in \cite{sah:dupont}. It is an abelian group defined by a presentation in terms of generators and relations and it  was shown by the authors to be closely related to $\kind{F}=\ho{3}{\spl{2}{F}}{\Z}$ when $F$ is algebraically closed. Soon afterwards Suslin proved (\cite[Theorem 5.2]{sus:bloch}) that the connection to $\kind{F}$ persists for all infinite fields $F$  (see Theorem \ref{thm:suslin} below). 
However, to derive an analagous result for $\ho{3}{\spl{2}{F}}{\Z}$ for general fields it is necessary to factor in the action of the multiplicative group of the field. The refined scissors congruence group 
$\rpb{F}$ of the 
field $F$ -- introduced in \cite{hut:cplx13} -- is defined by generators and relations analagously to the scissors congruence group but as a module over $\Z[F^\times]$ and not merely an abelian group.  
It can then be shown to bear approximately the same relation to $\ho{3}{\spl{2}{F}}{\Z}$ as $\pb{F}$ has to $\kind{F}$. (For a precise statement, see Theorem \ref{thm:rbl} below.)  Using some later results of the author about refined scissors congruence groups, our starting point in this article is essentially a presentation of $\aug{\Q}\ho{3}{\spl{2}{\Q}}{\zhalf{\Z}}$ as a module over the group ring $\Z[\Q^\times/(\Q^\times)^2]$ as well as the existence of module homomorphisms $S_p:\aug{\Q}\ho{3}{\spl{2}{\Q}}{{\Z}}\to \pb{\F{p}}$ (where the target is a module via $a\cdot x= (-1)^{v_p(a)}x$ for $a\in \Q^\times$), one for each prime $p$. 

\begin{rem} In our main theorem, we prove that the module homomorphism\\
 $\aug{\Q}\ho{3}{\spl{2}{\Q}}{\Z}\to \bigoplus_p \pb{\F{p}}$ induced by the maps $S_p$, ranging over all primes $p$, becomes an isomorphism after tensoring with $\zhalf{\Z}$. It is natural to ask whether the original homomorphism is an isomorphism over $\Z$. 

I do not know. Our methods of proof and $2$-torsion ambiguities in existing results require us to work over $\zhalf{\Z}$. However, it is not hard to show even over $\Z$ that the cokernel of this map is annihilated by $4$. 
\end{rem}
\begin{rem} It is to be expected that some version of the main result should hold for general number fields and even global fields. In order to arrive at such a result it would appear necessary first to 
determine whether the action of the (square classes of) the global units is trivial on the groups $\ho{3}{\spl{2}{F}}{\zhalf{\Z}}$. There is some mild evidence suggesting that this is so: (i) for any field the square class $\an{-1}$ acts 
trivially (see Theorem \ref{thm:rpbp} below)  and (ii) for local fields with finite residue field, the units act trivially. (This follows, for example, from Corollary \ref{cor:compdv} below.)  We hope to examine these questions elsewhere. 
\end{rem}

\begin{rem} There ought also to be analogous results for geometric function fields, at least over algebraically closed, or quadratically closed, fields.

 For example,
let $\hot{\C(x)}{\Z}$ denote the kernel of $\ho{3}{\spl{2}{\C(x)}}{\Z}\to \kind{\C(x)}$. There is a natural surjective homomorphism of $\Z[\C(x)^\times]$-modules
\[
\hot{\C(x)}{\Z}\to \bigoplus_{p\in \projl{\C}}\pb{\C}
\]
where the action of $\C(x)^\times$ on the component   $\pb{\C}$ indexed by a given $p$ on the right is given by $f\cdot x:= (-1)^{v_p(f)}x$. By analogy with our main theorem below, it is natural to ask whether this map is an isomorphism. (The group $\pb{\C}$ is known to be a $\Q$-vector space and one would expect the left-hand side also to be uniquely divisible, so that the result should hold without the need to invert $2$.)  
\end{rem} 

\textbf{Acknowledgements.} I thank the referee for a very careful and thorough reading of the article, and in particular for identifying a gap (now filled) in the proof of Theorem \ref{thm:compdv}. 

\subsection{Layout of the article}

In Section \ref{sec:rsc} we review some of the relevant known results about scissors congruence groups and their relation to the third homology of $\mathrm{SL}_2$ of fields. We introduce here the module $\rpbp{F}$ associated to a field $F$, which coincides with module $\rpbker{F}$ on tensoring with $\zhalf{\Z}$, but has the advantage of being a quotient rather than a submodule of $\rpb{F}$, and thus is defined by a presentation.  Our main results in the article depend on computations in $\rpbp{F}$. 

 In Section \ref{sec:val}, we use the algebraic properties of the refined scissors congruence groups to calculate $\ho{3}{\spl{2}{F}}{\zhalf{\Z}}$ for fields 
$F$ which are complete with respect to a discrete valuation. The results of this section give an update and a  strengthening of the main results of \cite{hut:sl2dv}.

In Section \ref{sec:Q} we specialize to the case of the field $\Q$ and state the main theorem.

Section \ref{sec:main} contains the proof of the main theorem (Theorem \ref{thm:main}) using the results and methods outlined in Sections \ref{sec:val} and 
\ref{sec:Q}.

In Section \ref{sec:app}, we describe some further applications of the main theorem; for example, the calculation of $\ho{3}{\spl{2}{\laur{\Q}{t}}}{\zhalf{\Z}}$ and an explicit description of a basis for the $\F{3}$-vector space elements of order dividing $3$ in $\ho{3}{\spl{2}{\Q}}{\Z}$. 

\subsection{Notation}
For a commutative unital ring $R$, $\unitr{R}$ denotes the group of units of $R$. 

For any abelian group $A$, we denote $A\otimes \znth{\Z}{n}$ by $\znth{A}{n}$. For any prime $p$, $\ptor{A}{p}$ denotes the vector space  $ \{ a\in A |\ pa=0\}$, of elements of order dividing $p$ in $A$. 

If $q$ is a prime power, $\F{q}$ will denote the finite field with $q$ elements.

For a group $G$ and a $\Z[G]$-module $M$, $M_G$ will denote the module of coinvariants; 
$M_G=\ho{0}{G}{M}=M/\aug{G}{M}$, where $\aug{G}$ is the augmentation ideal of $\Z[G]$.
 
Given an abelian group $G$ we let  $\asym{2}{\Z}{G}$ denote the group 
\[
\frac{G\otimes_{\Z}G}{<x\otimes y + y\otimes x | x,y \in G>}
\]
and, for $x,y\in G$, we denote by $x\asymm y$ the image of $x\otimes y$ in $\asym{2}{\Z}{G}$.

For any rational prime $p$, $v_p:\Q^\times\to \Z$ denotes the corresponding discrete valuation, determined by $a=p^{v_p(a)}\cdot (m/n)$ with $m,n$ not divisible by $p$. 

For a field $F$, we let $\sgr{F}$ denote the group ring $\Z[\sq{F}]$ of the group 
of square classes of $F$ and we let $\aug{F}$ denote 
the augmentation ideal of $\sgr{F}$. If $x\in F^\times$, we denote the corresponding square-class, 
considered as an element of $\sgr{F}$, by $\an{x}$. The generators $\an{x}-1$ of $\aug{F}$ will be 
denoted $\pf{x}$.
\section{Refined scissors congruence groups}\label{sec:rsc}

In this section we review some of the relevant known facts about the third homology of $\mathrm{SL}_2$ of fields and its description in terms of refined scissors congruence groups. 
\subsection{Indecomposable $K_3$}

For any field $F$ there is a natural surjective homomorphism
\begin{eqnarray}\label{eqn:kind}
\hoz{3}{\spl{2}{F}}\to \kind{F}.
\end{eqnarray}

When $F$ is quadratically closed (i.e. when $F^\times= (F^\times)^2$) this map is an isomorphism. 
However, in general, the group extension 
\[
1\to \spl{2}{F}\to\gl{2}{F}\to F^\times\to 1
\]
induces an action -- by conjugation -- of $F^\times$ on $\hoz{\bullet}{\spl{2}{F}}$ which factors 
through $\sq{F}$. It 
can be shown that the map (\ref{eqn:kind}) is a homomorphism of $\sgr{F}$-modules (where $\sq{F}$ acts trivially on $\kind{F}$) and induces an isomorphism
\begin{eqnarray}\label{iso}
\ho{3}{\spl{2}{F}}{\zhalf{\Z}}_{\sq{F}}\cong\zhalf{\kind{F}}
\end{eqnarray}
(see \cite[Proposition 6.4]{mirzaii:third}), but -- as our calculations in \cite{hut:rbl11}  show -- the action 
of $\sq{F}$ on 
$\hoz{3}{\spl{2}{F}}$ is in general non-trivial. 

 Let 
$
\hot{F}{\Z}$ denote the kernel of the surjective homomorphism ${\ho{3}{\spl{2}{F}}{\Z}\to\kind{F}}$. This is an $\sgr{F}$-submodule of $\ho{3}{\spl{2}{F}}{\Z}$. Note that the isomorphism (\ref{iso}) implies that 
\[
\hot{F}{\zhalf{\Z}}=\aug{F}\ho{3}{\spl{2}{F}}{\zhalf{\Z}}.
\]

\begin{rem} When $F$ is a number field the surjective homomorphism  
$\ho{3}{\spl{2}{F}}{\Z}\to\kind{F}$ is split as a map of $\Z$-modules. In fact, $\kind{F}$ is a finitely generated abelian group and it is enough to show that there is a torsion subgroup of 
$\ho{3}{\spl{2}{F}}{\Z}$ mapping isomorphically to the (cyclic) torsion subgroup of $\kind{F}$. But this latter statement follows from the explicit calculations of C. Zickert in \cite[Section 8]{zick:ext}. 
It follows that, as an abelian group, 
\[
\ho{3}{\spl{2}{F}}{\Z}\cong \kind{F}\oplus \hot{F}{\Z}
\]
for any number field $F$. 

However, there is no such decomposition of $\ho{3}{\spl{2}{F}}{\Z}$ as
an $\sgr{F}$-module. For details, see Remark \ref{rem:split} below. 
\end{rem}

\subsection{Scissors Congruence Groups}

For a field $F$, with at least $4$ elements, the \emph{scissors congruence group} 
(also called the  \emph{pre-Bloch group}), $\pb{F}$, is the group 
generated 
by the elements $\gpb{x}$, $x\in F^\times$,  subject to the relations
\[
R_{x,y}:\ 0=  \gpb{x}-\gpb{y}+\gpb{\frac{y}{x}}-\gpb{\frac{1-x^{-1}}{1-y^{-1}}}+\gpb{\frac{1-x}{1-y}} \quad x,y\not= 1.
\] 

The map 
\[
\lambda:\pb{F}\to \asym{2}{\Z}{F^\times},\quad  [x]\mapsto \left(1-{x}\right)\asymm {x}
\]
is well-defined, and the \emph{Bloch group of $F$}, $\bl{F}\subset \pb{F}$, is 
defined to be the kernel of $\lambda$. 

For the fields with $2$ and $3$ elements  the following definitions allow us to include 
these fields in the statements of some of our results:

$\pb{\F{2}}=\bl{\F{2}}$ is a cyclic group of order $3$ with generator denoted $\bconst{\F{2}}$.  We let $\gpb{1}:=0$ in $\pb{\F{2}}$.

$\pb{\F{3}}$ is cyclic of order $4$ with generator $\gpb{-1}$. We have $\gpb{1}:=0$ in $\pb{\F{3}}$. $\bl{\F{3}}$ is the subgroup generated by $2\gpb{-1}$.  

We recall (see, for example, \cite[Lemma 7.4]{hut:cplx13}):

\begin{lem}\label{lem:pfin} If $q$ is a prime power then  $\bl{\F{q}}$ is cyclic of order
$(q+1)/2$ when $q$ is odd and $q+1$ when $q$ is even. 
\end{lem}

The following corollary is particularly relevant to this article:

\begin{cor}\label{cor:pfin}
If $q$ is a prime power then  $\zhalf{\pb{\F{q}}}$ is cyclic of order
$\oddpart{q+1}$. 
\end{cor}

The Bloch group is closely related to the indecomposable $K_3$ of the field $F$:

\begin{thm}\label{thm:suslin}
For any field $F$ there is a natural exact sequence 
\[
0\to \widetilde{\Tor{\mu_F}{\mu_F}}\to \kind{F}\to \bl{F} \to 0
\]
where $ \widetilde{\Tor{\mu_F}{\mu_F}}$ is the unique nontrivial extension of $\Tor{\mu_F}{\mu_F}$ 
by $\Z/2$.
\end{thm}
 (See 
Suslin \cite{sus:bloch} for infinite fields and \cite{hut:cplx13} for finite fields.)

\subsection{The refined scissors congruence group}
For a field $F$ with at least $4$ elements, $\rpb{F}$ is defined to be the 
  $\sgr{F}$-module with generators $\gpb{x}$, $x\in F^\times$ 
subject to the relations 
\[
S_{x,y}:\quad 0=\gpb{x}-\gpb{y}+\an{x}\gpb{ \frac{y}{x}}-\an{x^{-1}-1}\gpb{\frac{1-x^{-1}}{1-y^{-1}}}
+\an{1-x}\gpb{\frac{1-x}{1-y}},\quad x,y\not= 1.
\]

Of course, from the definition it follows immediately that 
\[
\pb{F}=(\rpb{F})_{\sq{F}}=\ho{0}{\sq{F}}{\rpb{F}}.
\]

Let  $\bw= (\lambda_1,\lambda_2)$ 
be the $\sgr{F}$-module homomorphism 
\[
\rpb{F}\to \aug{F}^2\oplus\asym{2}{\Z}{F^\times}
\]
where 
$\lambda_1:\rpb{F}\to \aug{F}^2$ is the map $\gpb{x}\mapsto \pf{1-x}\pf{x}$, and $\lambda_2$ is the composite
\[
\xymatrix{
\rpb{F}\ar@{>>}[r]
&\pb{F}\ar[r]^-{\lambda}
&\asym{2}{\Z}{F^\times}.
}
\] 

It can be shown that $\bw$ is well-defined. 

The \emph{refined scissors congruence group} of $F$ (when $F$ has at least $4$ elements)  is the $\sgr{F}$-module
$
\rpbker{F}:=\ker{\lambda_1}
$.

The  
\emph{refined Bloch group} of the field $F$ (with at least $4$ elements) to be the $\sgr{F}$-module 
\begin{eqnarray*}
\rbl{F}:&=&\ker{\bw: \rpb{F}\to \aug{F}^2\oplus\asym{2}{\Z}{F^\times}}\\
&=&\ker{\lambda_2:\rpbker{F}\to \asym{2}{\Z}{F^\times}}.\\
\end{eqnarray*}

We can also define appropriate notions for the fields with $2$ and $3$ elements as follows:

$\pb{\F{2}}=\rpb{\F{2}}=\rbl{\F{2}}$ 
is simply an additive group of order $3$ with distinguished generator, denoted $\bconst{\F{2}}$.

$\rpb{\F{3}}$ is the cyclic $\sgr{\F{3}}$-module generated by the symbol $\gpb{-1}$ and subject to the one relation
\[
0=2\cdot(\gpb{-1}+\an{-1}\gpb{-1}).
\] 
$\pb{\F{3}}=\ho{0}{\F{3}^\times}{\rpb{\F{3}}}$ is then cyclic of order $4$ generated by the symbol $\gpb{-1}$.  $\rbl{\F{3}}$ is the submodule of order $2$ in $\rpb{\F{3}}$ generated by
$\gpb{-1}+\an{-1}\gpb{-1}$.

The symbol $\gpb{1}$ continues to denote $0$ in $\rpb{\F{2}}$ and $\rpb{\F{3}}$. 

We recall some results from \cite{hut:cplx13}: The main result there is 

\begin{thm}\label{thm:rbl} Let $F$ be any field.

There is a natural complex 
\[
0\to \Tor{\mu_F}{\mu_F}\to\ho{3}{\spl{2}{F}}{{\Z}}\to{\rbl{F}}\to 0
\]
which is exact everywhere except possibly at the middle term. The middle homology is annihilated by $4$.

In particular, for any field there is a natural short exact sequence
\[
0\to\zhalf{\Tor{\mu_F}{\mu_F}}\to\ho{3}{\spl{2}{F}}{\zhalf{\Z}}\to\zhalf{\rbl{F}}\to 0.
\]
\end{thm}

\subsection{Scissors congruence groups  and $
\hot{F}{\Z}$}
In \cite{sus:bloch} Suslin defines the elements $\sus{x}:=\sgpb{x}+\sgpb{x^{-1}}\in \pb{F}$ and 
shows that they 
satisfy 
\[
\sus{xy}=\sus{x}+\sus{y} \mbox{ and }     2\sus{x}=0  \forall  x,y\in F^\times.
\]
In particular, $\sus{x}=0$ in $\zhalf{\pb{F}}$.

There are two natural liftings of these elements to $\rpb{F}$: given $x\in F^\times$ we define 
\[
\suss{1}{x}:=\sgpb{x}+\an{-1}{\sgpb{x^{-1}}}
\]
and 
\[
\suss{2}{x}:=\left\{
\begin{array}{ll}
\an{1-x}\left(\an{x}\sgpb{x}+\sgpb{x^{-1}}\right),& x\not= 1\\
0,& x=1
\end{array}
\right.
\]

(If $F=\F{2}$, we interpret this as $\suss{i}{1}=0$ for $i=1,2$. For $F=\F{3}$, we have $\suss{1}{-1}=\suss{2}{-1}
=\gpb{-1}+\an{-1}\gpb{-1}$. )
 
The maps $F^\times \to \rpb{F}, x \mapsto \suss{i}{x}$ are $1$-cocyles:   $\suss{i}{xy}=\an{x}\suss{i}{y}+\suss{i}{x}$ for all $x,y\in F^\times$. (See \cite[Section 3]{hut:rbl11}). In general, the elements 
$\psi_i(x)$ have infinite order however. 

We define $\qrpb{F}$ to be $\rpb{F}$ modulo the submodule generated by the elements $\suss{1}{x}$, $x\in F^\times$. Likewise, $\qpb{F}$ is the group $\pb{F}$ modulo the subgroup generated by the 
elements  $\sus{x}$, $x\in F^\times$. Note that since the elements $\sus{x}$ are annihilated by $2$, we have $\zhalf{\pb{F}}=\zhalf{\qpb{F}}$. 

For any field there is natural homomorphism of $\sgr{F}$-modules $\ho{3}{\spl{2}{F}}{\Z}\to\rpb{F}$ and we have (\cite[Corollary 2.8, Corollary 4.4]{hut:sl2dv}):

\begin{thm}\label{thm:rpbp}
For any field $F$,  the map $\ho{3}{\spl{2}{F}}{\Z}\to\rpb{F}$ induces an isomorphism of $\sgr{F}$-modules
\[
\hot{F}{\zhalf{\Z}}=\aug{F}\ho{3}{\spl{2}{F}}{\zhalf{\Z}}\cong \aug{F}\zhalf{\rpbker{F}}
\]
and furthermore
\[
\zhalf{\rpbker{F}}= \zhalf{\qrpbker{F}}=
\ep{-1} \zhalf{\qrpb{F}}
\]
where $\ep{-1}$ denotes the idempotent $(1+\an{-1})/2\in \zhalf{\sgr{F}}$. 
\end{thm}

Note that it follows that the square class $\an{-1}$ acts trivially on $\ho{3}{\spl{2}{F}}{\zhalf{\Z}}$.

To simplify the right-hand side we define the module $\rpbp{F}$ to be $\qrpb{F}$ modulo the submodule generated by the elements $(1-\an{-1})\gpb{x}$, $x\in F^\times$. Thus 
$\rpbp{F}$ is the $\sgr{F}$-module generated by the the elements $\gpb{x},x\in F^\times$ subject to the relations 
\begin{enumerate}
\item $\gpb{1}=0$
\item $S_{x,y}=0$ for $x,y\not=1$
\item $\an{-1}\gpb{x}=\gpb{x}$ for all $x$.
\item $\sgpb{x^{-1}}=-\gpb{x}$  for all $x$
\end{enumerate}

The theorem then says that the 
map $\ho{3}{\spl{2}{F}}{\Z}\to\rpb{F}$ induces an isomorphism 
\[
\hot{F}{\zhalf{\Z}}\cong \aug{F}\zhalf{\rpbp{F}}. 
\]

Note that the natural map $\rpbp{F}\to \qpb{F}$ induces an isomorphism $\rpbp{F}_{F^\times}\cong \qpb{F}$. Furthermore, the results of \cite[Section 7]{hut:cplx13} immediately imply that 
$k^\times$ acts trivially on $\rpbp{k}$ when $k$ is a finite field. Thus $\rpbp{k}=\qpb{k}$ for a finite field $k$.  

\subsection{Some algebra in $\rpb{F}$}
 For any field $F$ the elements $C(x):=\gpb{x}+\an{-1}\gpb{1-x}+\pf{1-x}\suss{1}{x}\in \rpb{F}$, $x\in F^\times\setminus\{ 1\}$ can be shown to be independent of 
$x$  (see \cite[Section 3.2]{hut:rbl11}). We denote this constant, as well as its image in any quotient of $\rpb{F}$,  by $\bconst{F}$. 

We review some of the fundamental properties of the element $\bconst{F}\in \rpb{F}$ (for proofs see \cite[Section 3.2]{hut:rbl11}). 
\begin{prop}\label{prop:bconst}
The element $\bconst{F}\in \rpb{F}$ has the following properties:
\begin{enumerate}
\item $3\cdot\bconst{F}=\suss{1}{-1}$  and $6\cdot \bconst{F}=0$. 
\item $2\pf{x}\bconst{F}=\suss{1}{x}-\suss{2}{x}$ for all $x\in F^\times$.
\item $2\cdot \bconst{F}=0$ if and only if $T^2-T+1$ splits in $F$. 
\end{enumerate}
\end{prop}

\begin{cor}\label{cor:bconst} For any field $F$, we have $3\cdot \bconst{F}=0$ in $\qrpb{F}$ and 
\[
\pf{x}\bconst{F}= \suss{2}{x}= \an{x-1}\pf{-x}\gpb{x} \mbox{ in } \qrpb{F}
\]
for all $x\in F^\times\setminus\{ 1\}$.
\end{cor}

\begin{proof} $3\cdot\bconst{F}=0$ in $\qrpb{F}$ since $\suss{1}{-1}=0$ in $\qrpb{F}$. Thus $ -\pf{x}\bconst{F}=2\cdot \pf{x} \bconst{F}=-\suss{2}{x}$ since 
$\suss{1}{x}=0$ in $\qrpb{F}$. Furthermore, in $\qrpb{F}$ we have 
\[
0=\suss{1}{x}=\gpb{x}+\an{-1}\gpb{x^{-1}} \implies \gpb{x^{-1}}=-\an{-1}\gpb{x}
\]
and hence 
\begin{eqnarray*}
\suss{2}{x}&=& \an{1-x}\left( \an{x}\gpb{x}+\gpb{x^{-1}}\right)\\
&=& \an{1-x}\left(\an{x}-\an{-1}\right)\gpb{x}\\
&=& \an{x-1}\pf{-x}\gpb{x}.\\
\end{eqnarray*}
\end{proof}

Observe that in $\qrpb{F}$ we have $\bconst{F}=\gpb{x}+\an{-1}\gpb{1-x}$ since $\suss{1}{x}=0$, and in $\rpbp{F}$ we have $\bconst{F}=\gpb{x}+\gpb{1-x}$ since 
$\an{-1}$ acts trivially by definition on $\rpbp{F}$.

It will be convenient below to introduce the following additional notation in $\rpbp{F}$: 
\[
\gpb{0}:= \bconst{F}\mbox{ and } \gpb{\infty}:=-\bconst{F}.  
\]
With this notation, we then have 
\[
\bconst{F}=\gpb{x}+\gpb{1-x} \mbox{ and } \suss{1}{x}=0 \mbox{ for all }x\in \projl{F}.
\]
\subsection{A character-theoretic local-global principle} 
We will use the following character-theoretic principles: 

Let $G$ be an abelian group satisfying $g^2=1$ for all $g\in G$. Let $\mathcal{R}$ denote the group ring ${\Z}[G]$. 

For a character $\chi\in \dual{G}:=\hmz{G}{\mu_2}$, 
let $\upp{\mathcal{R}}{\chi}$ be the ideal of 
$\mathcal{R}$ generated
by the elements $\{ g-\chi(g)\ |\ g\in G\}$. In other words $\upp{\mathcal{R}}{\chi}$ is the 
kernel of the ring 
homomorphism $\rho(\chi):\mathcal{R}\to \Z$ sending $g$ to $\chi(g)$ for any $g\in G$. 
We let $\dwn{\mathcal{R}}{\chi}$ denote the associated $\mathcal{R}$-algebra structure on $\Z$; i.e. 
$\dwn{\mathcal{R}}{\chi}:= \mathcal{R}/\upp{\mathcal{R}}{\chi}$. 

If $M$ is an $\mathcal{R}$-module, we let $\upp{M}{\chi}=
\upp{\mathcal{R}}{\chi}M $ and we let 
\[
\dwn{M}{\chi}:=M/\upp{M}{\chi}=
\left(\mathcal{R}/\upp{\mathcal{R}}{\chi}\right)\otimes_{\mathcal{R}} M
=\dwn{\mathcal{R}}{\chi}\otimes_{\mathcal{R}}  M.
\]

Thus $\dwn{M}{\chi}$ is the largest quotient module of $M$ with the property that 
$g\cdot m =\chi(g)\cdot m$ for all $g\in G$.

In particular, if $\chi=\chi_0$, the trivial character,
 then $\upp{\mathcal{R}}{\chi_0}$ is the augmentation ideal $\aug{G}$, 
$\upp{M}{\chi_0}=\aug{G}M$  and $\dwn{M}{\chi_0}=M_{G}$.

Given $m\in M$, $\chi\in \dual{G}$, we denote the image of $m$ in $\dwn{M}{\chi}$ by $\dwn{m}{\chi}$. For example, for any character $\chi\in \dual{F^\times/(F^\times)^2}$, we can give a presentation of the $\sgr{F}$-module $\dwn{\rpbp{F}}{\chi}$, which is our main object of study below, as follows:  It is the $\sgr{F}$-module with generators $\dwn{\gpb{x}}{\chi}$, $x\in F^\times$, subject to the relations 
\begin{enumerate}
\item $\an{a}\cdot \dwn{\gpb{x}}{\chi}:= \chi(a)\dwn{\gpb{x}}{\chi}$ for all $a,x\in F^\times$
\item $\dwn{\gpb{1}}{\chi}=0$
\item 
\[
0=\dwn{\gpb{x}}{\chi}-\dwn{\gpb{y}}{\chi}+\chi(x)\dwn{\gpb{ \frac{y}{x}}}{\chi}-\chi(x^{-1}-1)\dwn{\gpb{\frac{1-x^{-1}}{1-y^{-1}}}}{\chi}
+\chi(1-x)\dwn{\gpb{\frac{1-x}{1-y}}}{\chi}
\]
for all $x,y\not=1$
\item $\chi(-1)\dwn{\gpb{x}}{\chi}=\dwn{\gpb{x}}{\chi}$ for all $x$, and 
\item $\dwn{\gpb{x}}{\chi}=-\dwn{\gpb{x^{-1}}}{\chi}$ for all $x$. 
\end{enumerate}

We will need the following result (\cite[Section 3]{hut:sl2dv})
\begin{prop}\label{prop:localglobal}\ \\
\begin{enumerate}
\item For any $\chi\in \dual{G}$, $M\to \dwn{M}{\chi}$ is an exact functor on the category of $\zhalf{\mathcal{R}}$-modules.
\item Let $f:M\to N$ be an $\zhalf{\mathcal{R}}$-module homomorphism. For any $\chi\in\dual{G}$, let $\dwn{f}{\chi}:\dwn{M}{\chi}\to \dwn{N}{\chi}$. Then 
$f$ is bijective (resp. injective, surjective) if and only if $\dwn{f}{\chi}$ is bijective (resp. injective, surjective)  for all $\chi\in \dual{G}$.
\end{enumerate}
\end{prop}

\begin{cor}\label{cor:localglobal}  For any field $F$ and any $\chi_0\not=\chi\in \dual{F^\times/(F^\times)^2}$, the natural $\sgr{F}$-homomorphism 
$\ho{3}{\spl{2}{F}}{\Z}\to \rpbp{F}$ induces an isomorphism
\[
\dwn{\ho{3}{\spl{2}{F}}{\zhalf{\Z}}}{\chi}\cong \dwn{\zhalf{\rpbp{F}}}{\chi}.
\]
\end{cor}

\begin{proof}
Since $\chi\not=\chi_0$, there exists $x\in F^\times$ with $\chi(x)=-1$ and hence for any $\sgr{F}$-module $M$ we have $\dwn{\left(\zhalf{M_{F^\times}}\right)}{\chi}=0$.  
Applying the functor $\dwn{( - )}{\chi}$ to the exact sequence $0\to \aug{F}M\to M\to M_{F^\times}\to 0$ thus shows that $\dwn{\zhalf{M}}{\chi}=\dwn{\left(\aug{F}\zhalf{M}\right)}{\chi}$.  The stated result thus follows from the isomorphism of $\sgr{F}$-modules 
\[
\aug{F}\ho{3}{\spl{2}{F}}{\zhalf{\Z}}\cong \aug{F}\rpbp{F}
\]
(Theorem \ref{thm:rpbp}).
\end{proof}

The following lemma will play a central role in all that follows:
\begin{lem}\label{lem:chi}
Let $\chi\in \dual{F^\times/(F^\times)^2}$. If $x\in F^\times$ with $\chi(x)=-1$, then
\[
2\dwn{\gpb{x}}{\chi}= 2\chi(x-1)\bconst{F} \mbox{ in } \dwn{\rpbp{F}}{\chi}.
\]
\end{lem}

\begin{proof} From Corollary \ref{cor:bconst}, we have 
\[
\left( \chi(x)-1\right)\bconst{F}= \chi(x-1)\left(\chi(-x)-1\right)\dwn{\gpb{x}}{\chi}.
\]
 However, we can suppose that $\chi(-1)=1$, since otherwise $\dwn{\rpbp{F}}{\chi}=0$, and hence $\chi(-x)=\chi(x)=-1$. Thus, we obtain
\[
-2\cdot \bconst{F}= -2 \chi(x-1)\dwn{\gpb{x}}{\chi}.
\]
\end{proof}
\section{Fields with a valuation}\label{sec:val}
\subsection{Valuations and the modules $\kv{v}$}  Given a field $F$ and a (surjective) valuation $v:F^\times \to \Gamma$, where $\Gamma$ is a totally ordered additive abelian group, we let $\mathcal{O}_v:= \{ x\in F^\times\ |\ v(x)\geq 0\} \cup \{ 0\}$ be the associated valuation ring, with maximal ideal $\mathcal{M}_v=\{ x \in \mathcal{O}_v\ | \ v(x)\not=0\}$, group of units $U_v:= \mathcal{O}_v\setminus\mathcal{M}_v$ and residue field $k=k(v):= \mathcal{O}_v/\mathcal{M}_v$. 

Given an $\sgr{k}$-module $M$, we denote by $\ind{F}{k}M$ the $\sgr{F}$-module $\sgr{F}\otimes_{\Z[U_v]}M$ (noting that the ring $\Z[U_v]$ surjects naturally onto $\sgr{k}$ and maps naturally to $\sgr{F})$. 

We recall the following result (see \cite[Section 5]{hut:sl2dv}): 

\begin{lem}\label{lem:spec}
There is a natural homomorphism of $\sgr{F}$-modules $\spec{v}:\qrpb{F}\to \ind{F}{k}\qrpb{k}$ given by 
\[
\spec{v}(\gpb{x})=
\left\{
\begin{array}{ll}
1\otimes \gpb{\bar{x}},& v(x)=0\\
1\otimes \bconst{k},& v(x)>0\\
-\left( 1\otimes \bconst{k}\right), & v(x)<0.\\
\end{array}
\right.
\]
\end{lem}

Now let $\kv{v}\subset \qrpb{F}$ be the $\sgr{F}$-submodule generated by $\{ \gpb{u}\ | \ u\in U_1=U_{1,v}:= 1+ \mathcal{M}_v\subset U_v\}$. (Caution: This is a slightly different definition from that given in \cite{hut:sl2dv}.) 

The following is a refinement of \cite[Lemma 5.2]{hut:sl2dv}:

\begin{prop}\label{prop:kv} Given a valuation $v$ on the field $F$, there is a natural short exact sequence of $\sgr{F}$-modules 
\[
\xymatrix{
0\ar[r]&\kv{v}\ar[r]& \qrpb{F}\ar^-{\spec{v}}[r]&\ind{F}{k}\qrpb{k}\ar[r]&0.\\
}
\]
\end{prop}

\begin{proof}
Certainly, $\kv{v}\subset \ker{\spec{v}}$ and, since $\spec{v}$ is clearly surjective, it induces a surjective homomorphism of $\sgr{F}$-modules
\[
\spec{v}:\dwn{\qrpb{F}}{v}:= \frac{\qrpb{F}}{\kv{v}}\to \ind{F}{k}\qrpb{k}. 
\]
To prove the Proposition it will thus suffice to construct an $\sgr{F}$-module homomorphism $T_v:\ind{F}{k}\qrpb{k}\to \dwn{\qrpb{F}}{v}$ satisfying 
$T_v\circ \spec{v}=\mathrm{Id}_{\dwn{\qrpb{F}}{v}}$. 

We will require the following three lemmas:

\begin{lem}\label{lem:1}
If $v(x)\not=0$, then 
\[
\gpb{x}=
\left\{
\begin{array}{ll}
\bconst{F}, & v(x) >0\\
-\bconst{F}, & v(x)<0\\
\end{array}
\right.
\mbox{ in } \dwn{\qrpb{F}}{v}.
\]
\end{lem}
\textit{Proof of Lemma \ref{lem:1}:}  If $v(x)>0$, then $\bconst{F}=\gpb{x}+\an{-1}\gpb{1-x}=\gpb{x}$ in $\dwn{\qrpb{F}}{v}$ since $1-x\in U_1$.

 If 
$v(x)<0$, then $v(x^{-1})>0$ and $\bconst{F}=\an{-1}\bconst{F}=\an{-1}\gpb{x^{-1}}=-\gpb{x}$, since $0=\suss{1}{x}=\gpb{x}+\an{-1}\gpb{x^{-1}}$ in 
$\qrpb{F}$.

\begin{lem}\label{lem:2}
For all $x\in F^\times$, $u\in U_1$, $\gpb{x}=\gpb{xu}$ in $\dwn{\qrpb{F}}{v}$.
\end{lem}

\textit{Proof of Lemma \ref{lem:2}:} Since $\gpb{u}=0$ in $\dwn{\qrpb{F}}{v}$ for all $u\in U_1$, we may assume $x\not\in U_1$. Then 
\[
0= \gpb{x}-\gpb{xu}+\an{x}\gpb{u}-\an{x^{-1}-1}\gpb{u\cdot\frac{1-x}{1-xu}}+\an{1-x}\gpb{\frac{1-x}{1-xu}}.
\]

But this implies $\gpb{x}=\gpb{xu}$ in $\dwn{\qrpb{F}}{v}$ since $u, (1-x)/(1-xu)\in U_1$. 

\begin{lem}\label{lem:3}
For all $x\in F^\times$, $u\in U_1$ we have $\an{u}\gpb{x}=\gpb{x}$ in $\dwn{\qrpb{F}}{v}$.
\end{lem}

\textit{Proof of Lemma \ref{lem:3}:}  Let $u\in U_1$, $u\not=1$.  By Corollary \ref{cor:bconst}, $\pf{u}\bconst{F}=\an{u-1}\pf{-u}\gpb{u}=0$ in 
$\dwn{\qrpb{F}}{v}$; i.e.,  $\an{u}\bconst{F}=\bconst{F}$ in $\dwn{\qrpb{F}}{ v}$ for all $u\in U_1$. 

Now for $x\in F^\times$, $u\in U_1$ (and $x\not=1$, $xu\not=1$) we have the following sequence of identities in $\dwn{\qrpb{F}}{v}$:
\begin{eqnarray*}
\pf{-xu}\gpb{x}&=& \pf{-xu}\gpb{xu} \mbox{ by Lemma \ref{lem:2}}\\
&=& \an{xu-1}\pf{xu}\bconst{F} \mbox{ by Corollary \ref{cor:bconst}}\\
&=& \an{xu-1}\pf{x}\bconst{F} \mbox{ since $\pf{xu}-\pf{x}=\an{x}\pf{u}$}\\
&=& \an{\frac{xu-1}{x-1}}\an{x-1}\pf{x}\bconst{F}\\
&=& \an{x-1}\pf{x}\bconst{F} \mbox{ since $(xu-1)/(x-1)\in U_1$}\\
&=& \pf{-x}\gpb{x}\mbox{ by Corollary \ref{cor:bconst} again}
\end{eqnarray*}
and hence 
\[
0= \left( \pf{-xu}-\pf{-x}\right)\gpb{x}=\an{-x}\pf{u}\gpb{x}
\]
proving the Lemma. 

By Lemma \ref{lem:3}, the $\Z[U_v]$-action on $\dwn{\qrpb{F}}{v}$ descends to an $\sgr{k}$-module structure. Combining this with Lemma \ref{lem:2}, there is a well-defined $\sgr{k}$-module homomorphism 
\[
t_v:\qrpb{k}\to \dwn{\qrpb{F}}{v}, \gpb{\bar{u}}\mapsto \gpb{u}, \quad u\in U_v.
\]
Thus there is an induced $\sgr{F}$-module homomorphism
\[
T_v:\ind{F}{k}\qrpb{k}=\sgr{F}\otimes_{\Z[U_v]}\qrpb{k}\to \dwn{\qrpb{F}}{v}, \quad \an{a}\otimes x\mapsto \an{a}t_v(x).
\]

Now, by choosing $u\in U_v\setminus U_1$, and noting that then $\bconst{k}=\gpb{\bar{u}}+\an{-1}\gpb{1-\bar{u}}$,  we see that $T_v(1\otimes \bconst{k})=\bconst{F}$.  Hence if $v(x)\not=0$ we have $T_v(S_v(\gpb{x}))=\gpb{x}$ in 
$\dwn{\qrpb{F}}{v}$ by Lemma \ref{lem:1}.  On, the other hand, if $u\in U_v$, then $T_v(S_v(\gpb{u}))=\gpb{u}$ (using Lemma \ref{lem:2} again), so that
$T_v\circ S_v =\mathrm{id}_{\dwn{\qrpb{F}}{v}}$ as required. 
\end{proof}

Tensoring with $\zhalf{\Z}$, taking the $\ep{-1}$-component and using Theorem \ref{thm:rpbp} we deduce

\begin{cor}\label{cor:kv} There is a natural short exact sequence of $\sgr{F}$-modules 
\[
\xymatrix{
0\ar[r]&\ep{-1}\zhalf{\kv{v}}\ar[r]&\zhalf{\rpbp{F}}\ar^-{\spec{v}}[r]&\ind{F}{k}\zhalf{\rpbp{k}}\ar[r]&0.\\
}
\]
\end{cor}
\subsection{Discrete valuations and the specialization homomorphism}  

Suppose that $v:F^\times \to \Z$ is a discrete valuation on the field $F$ with residue field $k=k(v)$. Let $\chi_v:F^\times/(F^\times)^2\to \mu_2$ denote the associated character defined by $\chi_v(a)=(-1)^{v(a)}$. 

For an abelian group $M$, we let $\rfmod{M}{v}$ denote the $\sgr{F}$-module $\sgr{\chi_v}\otimes_\Z M$. Equivalently, we equip $M$ with the $\sgr{F}$-module structure $\an{a}m:= (-1)^{v(a)}m$. 

\begin{thm}\label{thm:dv}  Let $F$ be  a field with discrete valuation $v:F^\times \to \Z$ and residue field $k$. Then we have natural isomorphisms
\[
\xymatrix{
\dwn{\zhalf{\rpbp{F}}}{\chi_v}\ar_-{\spec{v}}^-{\cong}[r]&\dwn{\left( \ind{F}{k}{\zhalf{\rpbp{k}}}\right)}{\chi_v}\ar^-{\cong}[r]&\rfmod{\zhalf{\pb{k}}}{v}.\\
}
\]

\end{thm}

\begin{proof} By Corollary \ref{cor:kv}, to prove the left-hand isomorphism, we must prove that $\dwn{\left(\ep{-1}\zhalf{\kv{v}}\right)}{\chi_v}=0$; i.e., we must prove that $\dwn{\gpb{u}}{\chi_v}=0$ in $\zhalf{\rpbp{F}}$ for all $u\in U_1$. This, in turn, follows from 
\begin{lem}\label{lem:chiv}  $\dwn{\gpb{x}}{\chi_v}=\bconst{F}$ in $\dwn{\zhalf{\rpbp{F}}}{\chi_v}$ whenever $v(x)>0$. 
\end{lem}

For, given this lemma, if $u\in U_1$ then $v(1-u)>0$ and hence 
\[
\dwn{\gpb{1-u}}{\chi_v}=\bconst{F}= \dwn{\gpb{u}}{\chi_v}+\chi_v(-1)\dwn{\gpb{1-u}}{\chi_v}=\dwn{\gpb{u}}{\chi_v}+\dwn{\gpb{1-u}}{\chi_v}.
\]

\textit{Proof of Lemma \ref{lem:chiv}:} Suppose that $x\in F^\times$ with $v(x)>0$. If $v(x)$ is odd then $\chi_v(x)=-1$ and hence 
\[
\dwn{\gpb{x}}{\chi_v}=\chi(x-1)\bconst{F}=\bconst{F} \mbox{ in } \dwn{\zhalf{\rpbp{F}}}{\chi_v}
\]
by Lemma \ref{lem:chi}.

Suppose, on the other hand, that $v(x)=2k$ with $k\geq 1$. Let $\pi\in F^\times$ with $v(\pi)=1$. So $x= \pi^{2k}u$  for some $u\in U_v$. Let $y=\pi u$ and $z=
\pi^{1-2k}\in F$. Then in $\dwn{\zhalf{\rpbp{F}}}{\chi_v}$ we have 
\[
0= \dwn{\gpb{z}}{\chi_v}-\dwn{\gpb{y}}{\chi_v}+\chi_v(z)\dwn{\gpb{x}}{\chi_v}-\chi_v(1-z^{-1})\dwn{\gpb{\frac{y(1-z)}{z(1-y)}}}{\chi_v}
+\chi_v(1-z)\dwn{\gpb{\frac{1-z}{1-y}}}{\chi_v}.
\]
We can identify all of the terms ocurring (except for $\dwn{\gpb{x}}{\chi_v}$) using the case of odd valuation: 

$v(y)=1 \implies \dwn{\gpb{y}}{\chi_v}=\bconst{F}$.  

$v(z^{-1})=2k-1 \implies \dwn{\gpb{z^{-1}}}{\chi_v}=\bconst{F} \implies 
\dwn{\gpb{z}}{\chi_v}=-\bconst{F}$ since $\gpb{a}=-\gpb{a^{-1}}$ in $\rpbp{F}$.

\[
\frac{1-z}{1-y}=\frac{\pi^{2k-1}-1}{\pi^{2k-1}(1-\pi u)} \implies v\left(\frac{1-z}{1-y}\right)=-(2k-1)
\]
and thus
\[
\dwn{\gpb{\frac{1-z}{1-y}}}{\chi_v}=-\bconst{F}.
\]
Futhermore
\[
v\left(\frac{y(1-z)}{z(1-y)}\right)=v\left(x\cdot \frac{1-z}{1-y}\right)=2k+(1-2k)=1 \implies \dwn{\gpb{\frac{y(1-z)}{z(1-y)}}}{\chi_v}=\bconst{F}.
\]

Since $\chi_v(z)=-1=\chi_v(1-z)$ and $\chi_v(1-z^{-1})=1$, we therefore deduce 
\[
0=-\bconst{F}-\bconst{F}-\dwn{\gpb{x}}{\chi_v}-\bconst{F}+\bconst{F}
\]
proving Lemma \ref{lem:chiv} (since $3\cdot\bconst{F}=0$ in $\rpbp{F}$).

The second isomorphism of the theorem follows from the general calculation for any $\sgr{k}$-module $M$: 
\begin{eqnarray*}
\dwn{\left(\ind{F}{k}M\right)}{\chi_v}&=& \sgr{\chi_v}\otimes_{\sgr{F}}\left(\sgr{F}\otimes_{\Z[U]}M\right)\\
&\cong& \sgr{\chi_v}\otimes_{\Z[U]}M\\
&\cong& \sgr{\chi_v}\otimes_\Z M_U \mbox{ since $\Z[U]\to \sgr{F}\to \sgr{\chi_v}$ factors through augmentation}\\
&=& \sgr{\chi_v}\otimes_\Z M_{k^\times} = \rfmod{M_{k^\times}}{v}. 
\end{eqnarray*}
\end{proof}

\begin{rem}  The isomorphism $\dwn{\zhalf{\rpbp{F}}}{\chi_v}\cong \rfmod{\zhalf{\pb{k}}}{v}$ of Theorem \ref{thm:dv} is induced by the map of $\sgr{F}$-modules
\begin{eqnarray*}
\sspec{v}:\rpbp{F}&\to&\rfmod{\qpb{k}}{v}\\
\gpb{x}&\mapsto&\left\{
\begin{array}{rc}
\gpb{\bar{x}},& v(x)=0\\
\bconst{k},& v(x)>0\\
-\bconst{k},& v(x)<0.\\
\end{array}
\right.
\end{eqnarray*}
\end{rem}

\begin{rem}\label{rem:sv}
Observe that this map  makes sense when $k(v)=\F{2}$ or $\F{3}$.
\end{rem}
\subsection{Fields complete with respect to a discrete valuation}
Let $F$ be a field with discrete valuation $v:F^\times\to \Z$, maximal ideal $\mathcal{M}_v$ and residue field $k$. For $n\geq 1$, let 
$U_n$ denote the subgroup $1+\mathcal{M}_v^n$ of $U=U_v$. 

Observe that if $F$ is complete with respect to the valuation $v$ and if muliplication by $m$ is invertible on $k$, then $U_1^m=U_1$ since $U_n/U_{n+1}\cong k$ for all $n$. Taking $m=2$, we deduce that $U_1=U_1^2$ whenever $F$ is complete with residue characteristic not equal to $2$.  On the other hand, for the field 
$\Q_2$, one has $U_1^2=U_3$. More generally, if $F$ is any finite degree extension of $\Q_2$ it is easily seen that there exists some $n>0$ such that 
$U_n\subset U_1^2$. On the other hand, for the complete field $F=\F{2}((x))$,  $U_1^2$ has infinite index in $U_1$ and hence $U_n\not\subset U_1^2$ for all 
$n>0$.  

The following significantly improves  \cite[Theorem 6.1]{hut:sl2dv}:
\begin{thm}\label{thm:compdv}  Let $v:F^\times \to \Z$ be a discrete valuation on the field $F$. Suppose that there exists $n>0$ such that $U_n\subset U_1^2$. Then the homomorphism $\spec{v}:\rpbp{F}\to \ind{F}{k}{\rpbp{k}}$ induces an isomorphism of $\sgr{F}$-modules
\[
\aug{F}\zhalf{\rpbp{F}}\cong \aug{F}\left(\ind{F}{k}{\zhalf{\rpbp{k}}}\right).
\]
\end{thm}

\begin{proof} By Proposition \ref{prop:localglobal} and Corollary \ref{cor:kv}, this is equivalent to the statement $\dwn{\zhalf{\kv{v}}}{\chi}=0$ for all characters $\chi \in \dual{F^\times/(F^\times)^2}$ satisfying $\chi(-1)=1$ and $\chi\not=\chi_0$ (the trivial character). To see this, apply the exact functor $\dwn{(\ )}{\chi}$ to the exact sequence of Corollary \ref{cor:kv} and observe that for any $\chi$ and for any $\zhalf{\sgr{F}}$-module $M$ we have 
\[
\dwn{(\epm{-1}M)}{\chi}=
\left\{
\begin{array}{ll}
\dwn{M}{\chi},& \chi(-1)=1\\
0,&\chi(-1)=-1.\\
\end{array}
\right.
\]

 Now, if $\chi\not=\chi_0$ and $\chi|_U$ is trivial then necessarily $\chi=\chi_v$ and this is Theorem \ref{thm:dv}.

So we can suppose that there exists $u\in U$ with $\chi(u)=-1$. Since $U_n\subset U_1^2$ we have $\chi(u)=1$ for all $u\in U_n$. Let $n_0\geq 1$ be minimal such that 
$\chi(u)=1$ for all $u\in U_{n_0}$. Thus there exists $u\in U_{n_0-1}$ with $\chi(u)=-1$ (where $U_0:=U$). Since 
\[
\chi(1-u^{-1})=\chi\left(-\frac{1-u}{u}\right)=\chi(-1)\chi(u)\chi(1-u)=-\chi(1-u),
\]
replacing $u$ by $u^{-1}$ if necessary, we can suppose that $\chi(1-u)=-1=\chi(u)$. 
Let $p: =1-u$. So $\chi(p)=-1$ and $v(p)=n_0-1\geq 0$.

 Let $a\in F^\times$ with $v(a)>0$. We will prove that $\dwn{\gpb{a}}{\chi}=\bconst{F}$ (from which the required result follows as in the proof of Theorem \ref{thm:dv}):

First consider the case $\chi(a)=1$. 
 In $\dwn{\zhalf{\rpbp{F}}}{\chi}$ we have 
\[
0=\dwn{\gpb{p}}{\chi}-\dwn{\gpb{ap}}{\chi}-\dwn{\gpb{a}}{\chi}-\dwn{\gpb{aw}}{\chi}-\dwn{\gpb{w}}{\chi}
\] 
where $w:= (1-p)/(1-ap)$ (and using $\chi(p)=-1=\chi(1-p)$, $\chi(1-p^{-1})=1$). 

Since $\chi(p)=\chi(1-p)=-1$, we have $\dwn{\gpb{p}}{\chi}=\chi(1-p)\bconst{F}=-\bconst{F}$ by Lemma \ref{lem:chi}. Similarly, $\chi(ap)=\chi(a)\chi(p)=-1$ while 
$\chi(1-ap)=1$, since $1-ap\in U_{n_0}$, so $\dwn{\gpb{ap}}{\chi}=\bconst{F}$.

We have $\chi(w)=\chi(1-p)\chi(1-ap)=-1$. So $\dwn{\gpb{w}}{\chi}=\chi(1-w)\bconst{F}$. But 
\[
1-w=p\cdot\frac{1-a}{1-ap}\implies \chi(1-w)=-\chi(1-a).
\]
So $\dwn{\gpb{w}}{\chi}=-\chi(1-a)\bconst{F}$. 

Finally, $\chi(aw)=\chi(a)\chi(w)=-1$. So $\dwn{\gpb{aw}}{\chi}=\chi(1-aw)\bconst{F}$. But $1-aw=(1-a)/(1-ap)$ so that $\chi(1-aw)=\chi(1-a)$. We deduce 
\[
0=-\bconst{F}-\bconst{F}-\dwn{\gpb{a}}{\chi}-\chi(1-a)\bconst{F}+\chi(1-a)\bconst{F}
\]
which forces $\dwn{\gpb{a}}{\chi}=\bconst{F}$, as required.

We now comsider the case $\chi(a)=-1$. Then $\dwn{\gpb{a}}{\chi}=\chi(1-a)\bconst{F}$ in $\dwn{\zhalf{\rpbp{F}}}{\chi}$ by Lemma \ref{lem:chi}. If $\chi(1-a)=1$ this gives the required conclusion.

This leaves us with the case  that $\chi(a)=-1=\chi(1-a)$.  We have  $\dwn{\gpb{a}}{\chi}=-\bconst{F}$ by Lemma \ref{lem:chi}.  Consider again the identity 
\[
0=\dwn{\gpb{p}}{\chi}-\dwn{\gpb{ap}}{\chi}-\dwn{\gpb{a}}{\chi}-\dwn{\gpb{aw}}{\chi}-\dwn{\gpb{w}}{\chi}
\] 
in $\dwn{\zhalf{\rpbp{F}}}{\chi}$.  We have $\dwn{\gpb{a}}{\chi}=\dwn{\gpb{p}}{\chi}=-\bconst{F}$. Since  $\chi(w)=\chi(1-p)\chi(1-ap)=-1$. $\chi(1-w)=-\chi(1-a)=1$ we have 
$\dwn{\gpb{w}}{\chi}=\bconst{F}$ by Lemma \ref{lem:chi}. Furthermore, $\chi(aw)=1$ and $\chi(1-aw)=\chi(1-a)\chi(1-ap)=-1$ gives $\dwn{\gpb{1-aw}}{\chi}=\bconst{F}$ and hence 
$\dwn{\gpb{aw}}{\chi}=0$.  We conclude that $0=-\dwn{\gpb{ap}}{\chi}-\bconst{F}$ and hence 
\begin{eqnarray}\label{ap1}
\dwn{\gpb{ap}}{\chi}=-\bconst{F} \mbox{ in }\dwn{\zhalf{\rpbp{F}}}{\chi}. 
\end{eqnarray}
On the other hand,  note that $\chi(ap)=1$ and $v(ap)>0$ so that 
\begin{eqnarray}\label{ap2}
\dwn{\gpb{ap}}{\chi}=\bconst{F} \mbox{ in }\dwn{\zhalf{\rpbp{F}}}{\chi}
\end{eqnarray}
by the case $\chi(a)=1$ above. Comparing (\ref{ap1}) and (\ref{ap2}), we conclude that $\bconst{F}=0$ in $\dwn{\zhalf{\rpbp{F}}}{\chi}$ and so the required identity $\dwn{\gpb{a}}{\chi}=\bconst{F}$ holds in this case also. 
\end{proof}

\begin{cor}  Let $v:F^\times \to \Z$ be a discrete valuation on the field $F$. Let $\chi\not=\chi_0\in \dual{\sq{F}}$. Suppose that $\chi(-1)=1$  and  that there exists $n>0$ such that $\chi|_{U_n}=1$.
Then 
\begin{enumerate}
\item 
\[
\dwn{\zhalf{\kv{v}}}{\chi}=0.
\]
\item Suppose further that $\chi|_{U_1}\not=1$. Then $\bconst{F}=0$ in $\dwn{\zhalf{\rpbp{F}}}{\chi}$. 
\end{enumerate}
\end{cor}

\begin{proof}
\begin{enumerate}
\item  This follows from the proof of Theorem \ref{thm:compdv} since the condition $U_n\subset U_1^2$  is only used to ensure that $\chi|_{U_n}=1$ for any given $\chi$.
\item Suppose that there exists $u\in U_1$ such that $\chi(u)=-1$. Let $a=1-u$. replacing $u$ by $u^{-1}$ if necessary, we can suppose that $\chi(a)=-1$. Thus 
$v(a)>0$ and $\chi(a)=\chi(1-a)=-1$. By the last case considered in the proof of Theorem \ref{thm:compdv}, it follows that $\bconst{F}=0$ in $\dwn{\zhalf{\rpbp{F}}}{\chi}$. 
\end{enumerate}
\end{proof}

\begin{cor}\label{cor:compdv}
Let $v:F^\times \to \Z$ be a discrete valuation on the field $F$ with residue field $k$. Suppose that there exists $n>0$ such that $U_n\subset U_1^2$. Then there is an isomorphism of 
$\sgr{k}$-modules
\[
\hot{F}{\zhalf{\Z}}\cong \hot{k}{\zhalf{\Z}}\oplus \zhalf{\rpbp{k}}.
\]
\end{cor}

\begin{proof}
By Theorem \ref{thm:compdv} 
\[
\hot{F}{\zhalf{\Z}}\cong \aug{F}\zhalf{\rpbp{F}}\cong \aug{F}\left(\zhalf{\ind{F}{k}\rpbp{k}}\right)
\]
(as $\sgr{F}$-modules).

By \cite[Lemma 5.4]{hut:sl2dv}, there is an $\sgr{k}$-module isomorphism
\[
\aug{F}\left(\zhalf{\ind{F}{k}\rpbp{k}}\right)\cong \aug{k}\left(\zhalf{\rpbp{k}}\right)\oplus\zhalf{\rpbp{k}}\cong \hot{k}{\zhalf{\Z}}\oplus \zhalf{\rpbp{k}}. 
\]
\end{proof}

\begin{cor}\label{cor:resfin} Let $F$ be a field with discrete valuation $v$ satisfying 
\begin{enumerate}
\item the residue field $k=k(v)$ is either finite or quadratically closed or real closed, and
\item there exists $n\geq 1$ such that $U_n\subset U_1^2$.
\end{enumerate}
Then we have natural isomorphisms of $\sgr{F}$-modules
\[
\xymatrix{
\hot{F}{\zhalf{\Z}}
\ar^-{\cong}[r]
&\dwn{\ho{3}{\spl{2}{F}}{\zhalf{\Z}}}{\chi_v}
\ar^-{\cong}_-{\sspec{v}}[r]
&\rfmod{\zhalf{\pb{k}}}{v}\\
}
\]
and thus there is a (split) exact sequence of $\sgr{F}$-modules
\[
0\to \rfmod{\zhalf{\pb{k}}}{v}\to \ho{3}{\spl{2}{F}}{\zhalf{\Z}}\to \zhalf{\kind{F}}\to 0. 
\]
\end{cor}
\begin{proof} Recall first that $\hot{\spl{2}{F}}{\zhalf{\Z}}\cong \aug{F}\zhalf{\rpbp{F}}$ as $\sgr{F}$-module by Theorem \ref{thm:rpbp}.
 
On the one hand, by Theorem \ref{thm:dv}, the map $\spec{v}$ induces an isomorphism of $\sgr{F}$-modules. 
\[
\dwn{\ho{3}{\spl{2}{F}}{\zhalf{\Z}}}{\chi_v}\cong \zhalf{\pb{k}}\{ v\}.
\]

On the other hand, Theorem \ref{thm:compdv} gives an isomorphism of $\sgr{F}$-modules
\[
\hot{\spl{2}{F}}{\zhalf{\Z}}\cong  \aug{F}\left(\ind{F}{k}{\zhalf{\rpbp{k}}}\right).
\]
 The conditions on the residue field $k$ imply that $\aug{k}\zhalf{\rpbp{k}}=0$ and hence that $\zhalf{\rpbp{k}}=\zhalf{\pb{k}}$ with trivial 
$\sgr{k}$-module structure.  Thus the result follows from:
\begin{lem}\label{lem:mv}
 Let $M$ be an $\sgr{k}$-module with trivial action of $k^\times$. Then
\[
\aug{F}\left(\zhalf{\ind{F}{k}M}\right) \cong \dwn{\left(\zhalf{\ind{F}{k}M}\right)}{\chi_v}\cong \rfmod{\zhalf{M}}{v}.
\]
\end{lem}

\textit{Proof of Lemma \ref{lem:mv}:}  By Proposition \ref{prop:localglobal}, to prove that the natural homomorphism, $S$ say,  
\[
\aug{F}\left(\zhalf{\ind{F}{k}M}\right) \to \dwn{\left(\aug{F}\left(\zhalf{\ind{F}{k}M}\right)\right)}{\chi_v}\cong \dwn{\left(\zhalf{\ind{F}{k}M}\right)}{\chi_v}
\]
is an isomorphism, it is enough to prove that $\dwn{S}{\chi}$ is an isomorphism for all $\chi\in \dual{F^\times/(F^\times)^2}$. 

When $\chi=\chi_0$ both domain and target of $\dwn{S}{\chi}$ are $0$. Likewise, if there exists $u\in U_v$ for which $\chi(u)=-1$, then (since $M$ has trivial $\Z[U]$-module structure by hypothesis), $\an{u}$ acts as multiplication by both $1$ and $-1$ on the target and domain, so that they vanish. This leaves only $\chi=\chi_v$, and $\dwn{S}{\chi_v}$ is tautologically an isomorphism. 
\end{proof}
\begin{rem} Note that the $\sgr{F}$-module direct sum decomposition in Corollary \ref{cor:resfin} is just the decomposition into $+1$ and $-1$-eigenspaces for the action of $\an{\pi}$ where $\pi$ is any element of $F$ with $v(\pi)=1$.
\end{rem}
\begin{exa}  Let $F$ be a local field with finite residue field. Suppose that either $F$ has characteristic $0$ or $\mathrm{char}(k)\not=2$. Then 
\[
\ho{3}{\spl{2}{F}}{\zhalf{\Z}}\cong \zhalf{\kind{F}}\oplus \rfmod{\zhalf{\pb{k}}}{v}
\]
as $\sgr{F}$-modules. In particular, for all primes $p$ we have 
\begin{eqnarray*}
\ho{3}{\spl{2}{\Q_p}}{\zhalf{\Z}}&\cong &\zhalf{\kind{\Q_p}}\oplus \hot{\Q_p}{\zhalf{\Z}}\\
&\cong & \zhalf{\kind{\Q_p}}\oplus \rfmod{\zhalf{\pb{\F{p}}}}{v}.\\
\end{eqnarray*}
\end{exa} 
\begin{exa}Consider the case $F=\laurs{\C}{x}$. Then we have 
\[
\ho{3}{\spl{2}{\laurs{\C}{x}}}{\zhalf{\Z}}\cong \zhalf{\kind{\laurs{\C}{x}}}\oplus \rfmod{{\pb{\C}}}{v}
\]
(since $\pb{\C}$ is a $\Q$-vector space). 
\end{exa}
\section{The field $\Q$} \label{sec:Q}
For a field $F$ with discrete valuation $v$, we let $\sspec{v}$ denote the composite $\sgr{F}$-module homomorphism 
\[
\ho{3}{\spl{2}{F}}{\zhalf{\Z}}\to \dwn{\left(\zhalf{\rpbp{F}}\right)}{\chi_v}\cong \rfmod{\zhalf{\pb{k(v)}}}{v}.
\]
(See remark \ref{rem:sv} above.) By abuse of notation, we will use the same symbol for the $\sspec{v}$ restricted to $\hot{F}{\zhalf{\Z}}$.

\begin{thm}\cite[Theorem 5.1]{hut:rbl11}
\label{thm:global}
Let $F$ be a field and let $\mathcal{V}$ be a family of discrete valuations on $F$ satisfying 
\begin{enumerate}
\item For any $x\in F^\times$, $v(x)=0$ for all but finitely many $v\in\mathcal{V}$.
\item  The map 
\[
F^\times\to\oplus_{v\in\mathcal{V}}\modtwo{\Z},\quad a\mapsto \{ v(a)\}_{v\in \mathcal{V}}
\]
is surjective.
\end{enumerate}
 Then the maps $\{ \sspec{v}\}_{v\in \mathcal{V}}$ induce a natural  surjective homomorphism 

\[
\hot{F}{\zhalf{\Z}}\cong  \aug{F}\zhalf{\rpbp{F}}\to \bigoplus_{v\in\mathcal{V}}\rfmod{\zhalf{\pb{k(v)}}}{v}.
\]
\end{thm}
\begin{rem} On the face of it, the collection of maps $\{\sspec{v}\}_{v\in \mathcal{V}}$ as above induces an $\sgr{F}$-module homomorphism with target the \emph{product} -- rather than the direct sum -- of the scissors congruence groups:
\[
\ho{3}{\spl{2}{F}}{\Z}\to \rpbp{F}\to \prod_{v\in \mathcal{V}} \rfmod{\pb{k(v)}}{v}.
\]
However, when we restrict to $\hot{F}{\Z}$ and tensor with $\zhalf{\Z}$ the image lies in the direct sum instead, in view of  the isomorphism $\hot{F}{\zhalf{\Z}}\cong \aug{F}\zhalf{\rpbp{F}}$ and the fact that $S_v(\pf{a}\gpb{b})=\pf{a}S_v(\gpb{b})=0$ whenever $v(a)$ is even. 
\end{rem}

Specializing to the case  $F=\Q$ and $\mathcal{V}=\primes$, the set of all primes, we obtain a surjective homomorphism $\bar{S}=\{ \sspec{p}\}_{p\in \primes}$ of $\sgr{\Q}$-modules
\begin{eqnarray}\label{map}
\hot{\Q}{\zhalf{\Z}}\to \bigoplus_{p\in\primes}\rfmod{\zhalf{\pb{\F{p}}}}{p}.
\end{eqnarray}

In the next section we will prove the following \emph{main theorem}:
\begin{thm}\label{thm:main}  The map 
\[
\bar{S}:\hot{\Q}{\zhalf{\Z}}\to \bigoplus_{p\in \primes}\rfmod{\zhalf{\pb{\F{p}}}}{p}
\] 
is an isomorphism of $\sgrp{\Q}$-modules.
\end{thm}

\begin{rem}
Since $\an{-1}\in \sgr{\Q}$ acts trivially on both of the modules in (\ref{map}), this is a map of $\sgrp{\Q}$-modules where 
\[
\sgrp{\Q}:= \Z[\Q^{\times}/\pm (\Q^\times)^2]=\Z[\Q_+/\Q_+^2].
\]
\end{rem}

\begin{rem}  We note that the image of the map 
\[
\bar{S}:\ho{3}{\spl{2}{\Q}}{\zhalf{\Z}}\to \prod_{p\in\primes}\rfmod{\zhalf{\pb{\F{p}}}}{p}
\]
does \emph{not} lie in the direct sum $\bigoplus_{p\in \primes}\rfmod{\zhalf{\pb{\F{p}}}}{p}$:

Let $t\in\spl{2}{\Z}$ be the element of order $3$ 
\[
t:= \left(
\begin{array}{ll}
-1&1\\
-1&0\\
\end{array}
\right)\in \spl{2}{\Z}.
\]

Denote (also) by $\bconst{\Q}\in \ho{3}{\spl{2}{\Q}}{\Z}$ the image of $1\in \Z/3=\ho{3}{\an{t}}{\Z}$ under the map induced by the inclusion $\an{t}\to \spl{2}{\Q}$. Then $\bconst{\Q}\in 
\ho{3}{\spl{2}{\Q}}{\Z}$ maps to $\bconst{\Q}\in \rpbp{\Q}$ (\cite[Remark 3.14]{hut:rbl11}). Note that $\bar{S}_p(\bconst{\Q})=\bconst{\F{p}}\in \qpb{\F{p}}$ for all primes $p$. Furthermore, $\bconst{\F{p}}\not=0$ 
precisely when $p\equiv 2\mod{3}$ (i.e., precisely when $3|p+1$), by \cite[Lemma 7.11]{hut:cplx13}.

In particular, the image of $\bconst{\Q}$ under the map $\{\bar{ S}_p\}_p$ lies in the product, but not the direct sum, of the scissors congruence groups of the residue fields. 
\end{rem} 
\begin{rem}
In view of Corollary \ref{cor:resfin} above, Theorem \ref{thm:main} can be stated equivalently as follows:
The natural homomorphism  $\ho{3}{\spl{2}{\Q}}{\Z}\to \prod_p\ho{3}{\spl{2}{\Q_p}}{\Z}$ induces an isomorphism
\[
\xymatrix{
\hot{\Q}{\zhalf{\Z}}\ar^-{\cong}[r]&\bigoplus_p \hot{\Q_p}{\zhalf{\Z}}.
}
\]
\end{rem}
\begin{rem}\label{rem:split}  We observe that -- unlike in the local case -- the short exact sequence of $\sgr{\Q}$-modules 
\[
0\to \hot{\Q}{\zhalf{\Z}}\to \ho{3}{\spl{2}{\Q}}{\zhalf{\Z}}\to \zhalf{\kind{\Q}}\to 0
\]
has no $\sgr{\Q}$-splitting (it is $\Z$-split, however):

Suslin's map gives a canonical isomorphism $\ptor{\kind{\Q}}{3}\cong\ptor{\bl{\Q}}{3}=\Z/3\cdot\bconst{\Q}\subset \qpb{\Q}$ and we can let $\bconst{\Q}$ also denote the corresponding element of $\kind{\Q}$. 

Recall that 
$\sgr{\Q}$ acts trivially on $\kind{\Q}$. Suppose that there were an $\sgr{\Q}$-module splitting $j:\zhalf{\kind{\Q}}\to \ho{3}{\spl{2}{\Q}}{\zhalf{\Z}}$. Then we would have $j(\bconst{\Q})=\bconst{\Q}+h$ for some 
$h\in \hot{\Q}{\zhalf{\Z}}$. We must have $\sgr{\Q}$ acts trivially on $j(\bconst{\Q})$ and hence $\pf{p}j(\bconst{\Q})=\pf{p}(\bconst{\Q}+h)=0$ for all primes $p$. However, we can  choose a prime $p$ such that 
$\pf{p}h=0$ in $\hot{\Q}{\zhalf{\Z}}$ and $p\equiv 2\pmod{3}$. Then $\bar{S}_p(\pf{p}(\bconst{\Q}+h))=\pf{p}\bconst{\F{p}}=-2\bconst{\F{p}}\not=0$, giving us a contradiction. So no such $\sgr{\Q}$-splitting $j$ can exist. 
\end{rem}
\section{Proof of the  Main Theorem}\label{sec:main}
In this section we prove Theorem \ref{thm:main}.

Recall now that $\sgrp{\Q}=\Z[G]$ where $G=\Q_+/\Q_+^2=\Q^\times/\pm(\Q^\times)^2$. As a multiplicative $\F{2}$-space, the set of all primes form a (number-theoretically) natural basis of  $\Q_+/\Q_+^2$.  Thus the space of characters 
$\dual{\Q_+/\Q_+^2}$ is naturally parametrised by the subsets of the set $\primes$ of positive prime numbers: if $S\subset \primes$ then the corresponding character $\chi_S$ is defined by 
\[
\chi_S(p)=
\left\{
\begin{array}{ll}
-1,& p\in S\\
1,& p\not\in S\\
\end{array}
\right.
\]
for all $p\in\primes$ or, equivalently, 
\[
\chi_S(x)=(-1)^{\sum_{p\in S}v_p(x)}
\]
for all $x\in \Q^{\times}$. Conversely, the character $\chi$ corresponds to the subset 
\[
\mathrm{supp}(\chi)
:= \{ p\in \primes\ |\ \chi(p)=-1\}. 
\]
(Thus, for a prime number $p$, $\chi_p$ is the unique character satisfying $\mathrm{supp}(\chi_p)=\{ p\}$.)

The following lemma is immediate from the definition of the $\sgr{\Q}$-module structure on $\rfmod{\pb{\F{p}}}{p}$.

\begin{lem} Let $\chi\in \dual{\Q_+/\Q_+^2}$. Let $p$ be a prime. Then
\[
\dwn{\rfmod{\zhalf{\pb{\F{p}}}}{p}}{\chi}=
\left\{
\begin{array}{ll}
\rfmod{\zhalf{\pb{\F{p}}}}{p},& \chi=\chi_p\\
0,& \mbox{ otherwise}\\
\end{array}
\right.
\]

\end{lem}


\begin{cor} For $\chi\in \dual{\Q_+/\Q_+^2}$ we have 
\[
\dwn{\left( \bigoplus_p{\rfmod{\zhalf{\pb{\F{p}}}}{p}}\right)}{\chi}=
\left\{
\begin{array}{ll}
\rfmod{\zhalf{\pb{\F{p}}}}{p},& \chi=\chi_p\mbox{ for some prime }p\\
0,& \mbox{ otherwise}\\
\end{array}
\right.
\]

\end{cor}

It thus follows from Proposition \ref{prop:localglobal} that to prove Theorem \ref{thm:main} it is enough to prove that, for any prime $p$, $\sspec{p}$ induces an isomorphism 
\[
\dwn{\zhalf{\rpbp{\Q}}}{\chi_p}\cong \rfmod{\zhalf{\pb{\F{p}}}}{p}
\]
for any prime $p$, while
\[
\dwn{\zhalf{\rpbp{\Q}}}{\chi}=0
\]
whenever $\mathrm{supp}(\chi)$ contains at least two distinct primes. The first of these statements is an immediate consequence of Theorem \ref{thm:dv} above. 
The second is  Corollary  \ref{cor:supp2} below.




\begin{lem}\label{lem:main}
Let $F$ be a field. Let $\chi\in \dual{\sq{F}}$. Suppose that $a\in F^\times$ satisfies  $\chi(1-a)=-1$. Then $\dwn{\gpb{a}}{\chi}= 0$ { in } $\dwn{\zhalf{\rpbp{F}}}{\chi}$.
\end{lem}

\begin{proof}If $\chi(-1)=-1$ we have $\dwn{\zhalf{\rpbp{F}}}{\chi}=0$. So we can suppose $\chi(-1)=1$.   In this case $\dwn{\gpb{1-a}}{\chi}=\bconst{F}$ by Lemma \ref{lem:chi}. But $\bconst{F}= \dwn{\gpb{a}}{\chi}+\dwn{\gpb{1-a}}{\chi}$ in $\dwn{\zhalf{\rpbp{F}}}{\chi}$.
\end{proof}

\begin{lem}\label{lem:ell}
 Let $F$ be a field. Let $\chi\in \dual{\sq{F}}$ with $\chi(-1)=1$. Let $\ell\in F^\times$ satisfy $\chi(\ell)=-1$ and $\chi(1-\ell)=1$. Then
\[
\dwn{\gpb{a}}{\chi}=\dwn{\gpb{(1-\ell)a}}{\chi} 
\]
{ in  } $\dwn{\zhalf{\rpbp{F}}}{\chi}$ for all $a\in \projl{F}$. 
\end{lem}

\begin{proof}  Observe that $\dwn{\gpb{1-\ell}}{\chi}=0$ by Lemma  \ref{lem:main}. In particular, the result holds for $a\in \{ 0,1,\infty\}$. 

For all $a\in F^\times\setminus\{ 1\}$ we have in $\dwn{\zhalf{\rpbp{F}}}{\chi}$
\[
0=\dwn{(S_{a^{-1},1-\ell})}{\chi}= \dwn{\gpb{a^{-1}}}{\chi}-\dwn{\gpb{1-\ell}}{\chi}+\chi(a^{-1})\dwn{\gpb{(1-\ell)a}}{\chi}-\chi(1-a)\dwn{\gpb{\frac{(1-\ell)(a-1)}{\ell}}}{\chi}+
\chi(1-a^{-1})\dwn{\gpb{\frac{a-1}{a\ell}}}{\chi}
\]
and hence
\begin{eqnarray}\label{eqn:ell}
0=\dwn{\gpb{a^{-1}}}{\chi}+\chi(a^{-1})\dwn{\gpb{y}}{\chi}-\chi(1-a)\dwn{\gpb{z}}{\chi}+\chi(1-a^{-1})\dwn{\gpb{w}}{\chi}
\end{eqnarray}
where
\[
y:=(1-\ell)a,\quad z:= \frac{(1-\ell)(a-1)}{\ell}\mbox{ and } w:= \frac{a-1}{a\ell}.
\]
Thus
\begin{eqnarray}\label{eqn:ell2}
1-z=\frac{1-y}{\ell}\mbox{ and } 1-w= \frac{1-y}{a\ell}.
\end{eqnarray}

We consider now the four possible values of $(\chi(a),\chi(1-a))$:

\begin{enumerate}
\item $\chi(a)=-1$ and $\chi(1-a)=1$. 

Then $\chi(a^{-1})=-1=\chi(1-a^{-1})$. Furthermore 
$\dwn{\gpb{a}}{\chi}=\bconst{F}$ { and } $\dwn{\gpb{a^{-1}}}{\chi}=-\bconst{F}$
by Lemma \ref{lem:chi}. 
By (\ref{eqn:ell}) we thus have
\[
0=-\bconst{F}-\dwn{\gpb{y}}{\chi}-\dwn{\gpb{z}}{\chi}-\dwn{\gpb{w}}{\chi}
\]
where $\chi(y)=-1=\chi(z)$ and $\chi(w)=1$. 

We divide further into sub-cases according to the value of $\chi(1-y)$: 
\begin{enumerate}
\item $\chi(1-y)=1$: Then $\dwn{\gpb{y}}{\chi}=\bconst{F}$ by Lemma \ref{lem:chi} and hence $\dwn{\gpb{y}}{\chi}=\dwn{\gpb{a}}{\chi}$ as required.
\item $\chi(1-y)=-1$: Then $\dwn{\gpb{y}}{\chi}=-\bconst{F}$ by Lemma \ref{lem:chi}. However, by (\ref{eqn:ell2}), $\chi(1-z)=\chi(1-y)\chi(\ell)=1$ and $\chi(1-w)=\chi(1-y)\chi(a\ell)=-1$ so that 
$\dwn{\gpb{z}}{\chi}=\bconst{F}$ and $\dwn{\gpb{w}}{\chi}=0$ by Lemmas \ref{lem:chi} and \ref{lem:main}.   Hence, by (\ref{eqn:ell}), we now have 
$
0=-\bconst{F}+\bconst{F}-\bconst{F}-0$ and hence $ \bconst{F}=0$ { in } $\dwn{\zhalf{\rpbp{F}}}{\chi}$.
Thus $\dwn{\gpb{y}}{\chi}=0=\dwn{\gpb{a}}{\chi}$ as required, in this case also.
\end{enumerate}
\item $\chi(a)=-1$ and $\chi(1-a)=-1$. 

Then $\chi(a^{-1})=-1$ and $\chi(1-a^{-1})=1$. Thus
$
\dwn{\gpb{a}}{\chi}=-\bconst{F}$ { and } $\dwn{\gpb{a^{-1}}}{\chi}=\bconst{F}$.
This gives $
0=\bconst{F}-\dwn{\gpb{y}}{\chi}+\dwn{\gpb{z}}{\chi}+\dwn{\gpb{w}}{\chi}$
where 
$\chi(y)=-1=\chi(w)$ and $\chi(z)=1$. 

\begin{enumerate}
\item $\chi(1-y)=1$: Then $\dwn{\gpb{y}}{\chi}=\bconst{F}=-\dwn{\gpb{a}}{\chi}$. However, by (\ref{eqn:ell2}) again, $\chi(1-z)=-1$ and $\chi(1-w)=1$ so that $\dwn{\gpb{z}}{\chi}=0$ and $\dwn{\gpb{w}}{\chi}=\bconst{F}$. 
From (\ref{eqn:ell}) we have 
$
0=\bconst{F}-\bconst{F}+0+\bconst{F}
$
and hence $\bconst{F}=0$ in $\dwn{\zhalf{\rpbp{F}}}{\chi}$ as required. 
\item $\chi(1-y)=-1$:  Then $\dwn{\gpb{y}}{\chi}=-\bconst{F}=\dwn{\gpb{a}}{\chi}$ again as required.
\end{enumerate}
\item $\chi(a)=1$ and $\chi(1-a)=-1$. 

Then
$
\dwn{\gpb{a}}{\chi}=0=\dwn{\gpb{a^{-1}}}{\chi}
$
by Lemma \ref{lem:main}. Thus from (\ref{eqn:ell}) we have 
$
0=\dwn{\gpb{y}}{\chi}+\dwn{\gpb{z}}{\chi}-\dwn{\gpb{w}}{\chi}
$
where $\chi(y)=\chi(z)=\chi(w)=1$. 

\begin{enumerate}
\item $\chi(1-y)=1$: Then $\chi(1-z)=-1=\chi(1-w)$. Hence 
$
\dwn{\gpb{z}}{\chi}=0=\dwn{\gpb{w}}{\chi}.
$
Thus $\dwn{\gpb{y}}{\chi}=0=\dwn{\gpb{a}}{\chi}$ as required. 
\item $\chi(1-y)=-1$: Then $\dwn{\gpb{y}}{\chi}=0=\dwn{\gpb{a}}{\chi}$ by Lemma \ref{lem:main}. 
\end{enumerate}
\item $\chi(a)=1=\chi(1-a)$.

 Then $\chi(a^{-1})=1=\chi(1-a^{-1})$ also. Equation (\ref{eqn:ell}) thus gives 
$
0=\dwn{\gpb{a^{-1}}}{\chi}+\dwn{\gpb{y}}{\chi}-\dwn{\gpb{z}}{\chi}+\dwn{\gpb{w}}{\chi}
$
with $\chi(z)=-1=\chi(w)$.  Furthermore $\chi(1-z)=-\chi(1-y)=\chi(1-w)$. Hence
$\dwn{\gpb{z}}{\chi}=\dwn{\gpb{w}}{\chi}=-\chi(1-y)\bconst{F}$
by Lemma \ref{lem:chi}. This gives 
\[
0= \dwn{\gpb{a^{-1}}}{\chi}+\dwn{\gpb{y}}{\chi}=-\dwn{\gpb{a}}{\chi}+\dwn{\gpb{y}}{\chi}
\]
as required. 
\end{enumerate}
\end{proof}

A straightforward induction gives: 
\begin{cor}\label{cor:ellz}
 Let $F$ be a field. Let $\chi\in \dual{\sq{F}}$ with $\chi(-1)=1$. Let $\ell\in F^\times$ satisfy $\chi(\ell)=-1$ and $\chi(1-\ell)=1$. Then
\[
\dwn{\gpb{a}}{\chi}=\dwn{\gpb{(1-\ell)^ma}}{\chi} \mbox{ in  } \dwn{\zhalf{\rpbp{F}}}{\chi}
\]
for all $a\in \projl{F}$ and all $m\in \Z$. 
\end{cor}
\begin{cor}\label{cor:ell}
 Let $F$ be a field. Let $\chi\in \dual{\sq{F}}$ with $\chi(-1)=1$. Let $\ell\in F^\times$ satisfy $\chi(\ell)=-1$ and $\chi(1-\ell)=1$. Then
\[
\dwn{\gpb{a}}{\chi}=\dwn{\gpb{a+t\ell}}{\chi} \mbox{ in  } \dwn{\zhalf{\rpbp{F}}}{\chi}
\]
for all $a\in F$ and all $t\in \Z$. 
\end{cor}

\begin{proof} In $\dwn{\zhalf{\rpbp{F}}}{\chi}$ we have 
\begin{eqnarray*}
\dwn{\gpb{a}}{\chi}&=& \dwn{\gpb{a(1-\ell)^{-1}}}{\chi}\mbox{ by Corollary \ref{cor:ellz}}\\
&=& \bconst{F}-\dwn{\gpb{1-\frac{a}{1-\ell}}}{\chi}\\
&=& \bconst{F}-\dwn{\gpb{(1-\ell)\left(1-\frac{a}{1-\ell}\right)}}{\chi}\mbox{ by Lemma \ref{lem:ell}}\\
&=& \bconst{F}-\dwn{\gpb{1-(a+\ell)}}{\chi}\\
&=& \dwn{\gpb{a+\ell}}{\chi}\\
\end{eqnarray*}
for any $a\in F$. 
\end{proof}

\begin{prop} \label{prop:supp}
Let $\chi\in \dual{\Q_+/\Q_+^2}$. If $\card{\mathrm{supp}(\chi)}\geq 2$ then 
\[
\dwn{\gpb{a}}{\chi}=\dwn{\gpb{a+t}}{\chi}
\]
 { in }
$\dwn{\zhalf{\rpbp{\Q}}}{\chi}$ for all $t\in \Z$ and $a\in \Q$.
\end{prop}

\begin{proof}
Let $p=\mathrm{min}(\mathrm{supp}(\chi))$. Then $\chi(p)=-1$ and $\chi(1-p)=\chi(p-1)=1$. So 
\[
\dwn{\gpb{a}}{\chi}=\dwn{\gpb{a+tp}}{\chi}
\]
for all $a\in \Q$ and $t\in \Z$ by Corollary \ref{cor:ell}. 

 Now let $q=\mathrm{min}(\mathrm{supp}(\chi)\setminus\{ p\})$. 

Suppose first that $p>2$. The either $q-1$ or $q+1$ is not divisble by $p$. If $p$ does not divide $q-1$ take $\ell=q$. Otherwise take $\ell=-q$. Then 
$\chi(\ell)=-1$ and $\chi(1-\ell)=1$ so that for all $a\in \Q$ $\dwn{\gpb{a}}{\chi}=\dwn{\gpb{a+t\ell}}{\chi}$ for all $t\in \Z$ and hence 
$\dwn{\gpb{a}}{\chi}=\dwn{\gpb{a+tq}}{\chi}$ for all $t\in \Z$. 

Thus for all $a\in \Q$ we have 
\[
\dwn{\gpb{a}}{\chi}=\dwn{\gpb{a+tp+sq}}{\chi} \mbox{ for all } t,s\in \Z
\]
proving the proposition in this case.

Suppose now that $p=2$. 

If $q\equiv 5\pmod{8}$ then $v_2(1-q)=2$ so that  $\chi(1-q)=1$ and we can take $\ell=q$ and argue as above. 

If $q\equiv 3\pmod{8}$ the corresponding argument applies with $\ell=-q$. 

If $q\equiv -1\pmod{8}$ we can take $\ell=3q$.  Then $\chi(\ell)=-1$ (since $q\not= 3$). Furthermore  we have $\ell-1\equiv 4\pmod{8}$ and 
\[
0<\frac{\ell-1}{4}<q.
\] 
This implies $\chi(\ell-1)=\chi(1-\ell)=1$ and we can conclude as before. 

Finally, if $q\equiv 1\pmod{8}$ we take $\ell =-3q$ and argue as in the previous case. 
\end{proof}
\begin{cor}\label{cor:supp2}
Let $\chi\in \dual{\Q_+/\Q_+^2}$ and suppose that  $\card{\mathrm{supp}(\chi)}\geq 2$. Then  $\dwn{\zhalf{\rpbp{\Q}}}{\chi}=0$.
\end{cor}

\begin{proof} We will show that $\dwn{\gpb{a}}{\chi}=0$ for all $a\in \Q$. By Proposition \ref{prop:supp} we have $\dwn{\gpb{a}}{\chi}=\dwn{\gpb{a+t}}{\chi}$ for all $t\in \Z$, $a\in \Q^\times$. 
It follows that $\dwn{\gpb{a}}{\chi}= \dwn{\gpb{1}}{\chi}=0$ if $a\in \Z$. Thus also $\dwn{\gpb{1/a}}{\chi}=0$ for all $a\in \Z\setminus\{ 0\}$. 

Note that it is enough to prove $\dwn{\gpb{a}}{\chi}=0$ for all $a>0$ (if necessary replacing $a$ by $a+t$ with $t\in \Z$ large). So let $a=r/s$ with $0< r,s\in \Z$. 
We proceed by induction on $h:=\mathrm{min}(r,s)$. The case $h=1$ has already been proved. Suppose now that $n\geq 1$ and the statement is known for $h\leq n$. Consider the case $h=n+1$.  Replacing $a$ by $1/a$ if necessary we can suppose 
$s<r$ and $s=n+1$. Then there exists $t\in \Z$ such that $0<r':=r-ts<s$. So 
\[
\dwn{\gpb{a}}{\chi}=\dwn{\gpb{a-t}}{\chi}=\dwn{\gpb{r'/s}}{\chi}
\]
where now $h=r'\leq n$ and we are done. 
\end{proof}

\section{Some related calculations}\label{sec:app}
\subsection{The module $\rpbp{\Q}$} The module $\rpbker{F}$ arises inevitably in the calculation of the third homology of $\spl{2}{A}$ for various rings $A$. For example, if $F$ is any infinite field we have (\cite[Theorem 8.1]{hut:laur})
\[
\ho{3}{\spl{2}{\laur{F}{t}}}{\zhalf{\Z}}\cong \ho{3}{\spl{2}{F}}{\zhalf{\Z}}\oplus\zhalf{\rpbker{F}}
\]
and there is a natural short exact sequence of $\sgr{F}$-modules (\cite[Theorem 7.4, Example 7.9]{hut:slr})
\[
0\to \ho{3}{\spl{2}{\pows{F}{t}}}{\zhalf{\Z}}\to \ho{3}{\spl{2}{\laurs{F}{t}}}{\zhalf{\Z}}\to \zhalf{\rpbker{F}}\to 0. 
\]
Furthermore, there is a natural short exact sequence
\[
0\to \aug{F}\zhalf{\rpbker{F}}\to \ho{3}{\spl{2}{\pows{F}{t}}}{\zhalf{\Z}}\to \zhalf{\kind{\pows{F}{t}}}\to 0. 
\]

As noted above, for any field $F$ the natural $\sgr{F}$-homomorphism $\rpbker{F}\to \rpbp{F}$  induces an isomorphism  
\[
\zhalf{\rpbker{F}}\cong\zhalf{\rpbp{F}}.
\]

In Theorem \ref{thm:main} above we have calculated $\aug{\Q}\zhalf{\rpbp{\Q}}$. This easily gives a computation of $\zhalf{\rpbp{\Q}}$.  Namely, for any field $F$ there is a short exact sequence of $\sgr{F}$-modules (see, for example,  
\cite[Lemma 2.7]{hut:sl2dv})
\[
0\to \aug{F}\zhalf{\rpbp{F}}\to \zhalf{\rpbp{F}}\to \zhalf{\pb{F}}\to 0
\]
(where, $\sgr{F}$ acts trivially on $\pb{F}$). Now, by definition, there is an exact sequence of abelian groups
\[
0\to \bl{\Q}\to \pb{\Q} \to \asym{2}{\Z}{\Q^\times}\to K_2(\Q)\to 0.
\] 
Tensoring with $\zhalf{\Z}$ and using the fact that $K_2(\Q)$ is a torsion $\Z$-module, we deduce that 
\[
\zhalf{\pb{\Q}}\cong \zhalf{\bl{\Q}}\oplus \zhalf{\asym{2}{\Z}{\Q^\times}}\cong \zhalf{\bl{\Q}}\oplus V\cong \Z/3\oplus V
\]
where $V=\zhalf{\asym{2}{\Z}{\Q^\times}}$ is a free $\zhalf{\Z}$-module of countable rank.  Furthermore the exact sequence 
\[
0\to \aug{\Q}\zhalf{\rpbp{\Q}}\to \zhalf{\rpbp{\Q}}\to \zhalf{\pb{\Q}}\to 0
\]
splits as a sequence of $\zhalf{\Z}$-modules since the subgroup $\Z/3\cdot \bconst{\Q}\subset \zhalf{\rpbp{\Q}}$ maps isomorphically to $\zhalf{\bl{\Q}}$. Thus, in view of Theorem \ref{thm:main} we have:
\begin{lem}\label{lem:rpbkerq}  As a $\zhalf{\Z}$-module, $\zhalf{\rpbp{\Q}}$ is a direct sum of an infinite torsion group and a free $\zhalf{\Z}$-module $V$ of countable rank. More particularly:
\begin{eqnarray*}
\zhalf{\rpbp{\Q}}&\cong & \left(\bigoplus_{p\in \primes}\zhalf{\pb{\F{p}}}\right)\oplus\zhalf{\bl{\Q}}\oplus \zhalf{\asym{2}{\Z}{\Q^\times}}\\
&\cong &\left(\bigoplus_{p\in \primes}\Z/\oddpart{p+1}\right)\oplus \Z/3 \oplus V.\\
\end{eqnarray*}
\end{lem}

\begin{cor}\label{cor:laurq}
As an abelian group we have 
\[
\ho{3}{\spl{2}{\laur{\Q}{t}}}{\zhalf{\Z}}\cong \left(\bigoplus_{p\in\primes} \Z/\oddpart{p+1}\right)^{\oplus 2}\oplus (\Z/3)^{\oplus 2}\oplus V. 
\]
\end{cor}
\subsection{The module $\cconstmod{\Q}$ and the $3$-torsion in $\ho{3}{\spl{2}{\Q}}{\Z}$} 

We let $\cconstmod{F}$ denote the $\sgr{F}$-submodule of $\rpbp{F}$ generated by $\bconst{F}$.  Note that $3\cdot \cconstmod{F}=0$; $\cconstmod{F}$ is an $\F{3}$-vector space. 

For any field $F$, let $\mathcal{H}=\mathcal{H}_F$ denote the $\sgr{F}$-submodule of $\ho{3}{\spl{2}{F}}{\Z}$ generated by the image of $\ho{3}{\spl{2}{\Z}}{\Z}$.

\begin{rem}
Since the $\sgr{F}$-module structure on $\ho{3}{\spl{2}{F}}{\Z}$ is induced from the  action of $\gl{2}{F}$ by conjugation on $\spl{2}{F}$, $\mathcal{H}=\mathcal{H}_F$ is just the 
subgroup $\sum_{g\in \gl{2}{F}}\ho{3}{\spl{2}{\Z}^g}{\Z}$ in $\ho{3}{\spl{2}{F}}{\Z}$; i.e. it is the subgroup of $\ho{3}{\spl{2}{F}}{\Z}$ generated by $\spl{2}{\Z}$ and its 
$\gl{2}{F}$-conjugates. 
\end{rem}

\begin{prop}
 Suppose that $\mathrm{char}(F)\not=3$ and 
$\zeta_3\not\in F$. 

Then the map $\ho{3}{\spl{2}{F}}{\Z}\to \rpbp{F}$ 
induces an isomorphism $\zhalf{\mathcal{H}}\cong \ptor{\mathcal{H}}{3}\cong \cconstmod{F}$.

\end{prop}

\begin{proof} As above, let 
\[
t:= 
\left(
\begin{array}{ll}
-1&1\\
-1&0\\
\end{array}
\right)\in \spl{2}{\Z}
\]
and let $G$ be the cyclic subgroup of order $3$ generated by $t$. By \cite[Remark 3.14]{hut:rbl11}, the composite map $\Z/3=\ho{3}{G}{\Z}\to \ho{3}{\spl{2}{F}}{\Z}\to \qrpb{F}$ sends $1$ to 
$\bconst{F}$ for any field $F$. 

We recall that $\ho{3}{\spl{2}{\Z}}{\Z}\cong\Z/12$.  Furthermore,  the inclusion $G\to \spl{2}{\Z}$ induces an isomorphism 
\[
\Z/3\cong \ho{3}{G}{\Z}\cong \ptor{\ho{3}{\spl{2}{\Z}}{\Z}}{3}\cong \ho{3}{\spl{2}{\Z}}{\zhalf{\Z}}.
\]
Thus $\ptor{\mathcal{H}}{3}\cong \zhalf{\mathcal{H}}$ maps onto $\cconstmod{F}$. 

On the other hand, the kernel of the map $\ho{3}{\spl{2}{F}}{\zhalf{\Z}}\to \zhalf{\rpbp{F}}$ is isomorphic to $\zhalf{\mu_F}$. In particular, if $\zeta_3\not\in F$, the induced map 
$\ptor{\mathcal{H}}{3}\to \cconstmod{F}$ is also injective.
\end{proof}

\begin{lem}\label{lem:bconstq} We have $\cconstmod{\Q}=\ptor{\ho{3}{\spl{2}{\Q}}{\Z}}{3}$ and 
\[
\cconstmod{\Q}\cong \left(\bigoplus_{p\equiv -1\mod{3}}\Z/3\right)\oplus \Z/3.
\]
\end{lem}

\begin{proof} Since $\zhalf{\kind{\Q}}\cong \zhalf{\bl{\Q}}=\Z/3\cdot \bconst{\Q}$ we have a ($\Z$-split) short exact sequence of $\sgr{\Q}$-modules
\[
0\to \aug{\Q}\cconstmod{\Q}\to \cconstmod{\Q}\to \Z/3\bconst{\Q}\to 0.
\]

Consider the composite map 
\[
\xymatrix{
\aug{\Q}\cconstmod{\Q}\ar@{^{(}->}[r]& \aug{\Q}\zhalf{\rpbp{\Q}}\ar^-{\cong}[r]&\bigoplus_{p}\zhalf{\pb{\F{p}}}
}
\]
where the right-hand arrow is an isomorphism by Theorem \ref{thm:main}. 
We finish by observing that $\zhalf{\pb{\F{p}}}\cong \Z/\oddpart{p+1}$ has no $3$-torsion except when $p\equiv -1\mod{3}$ and when  $p\equiv -1\mod{3}$ the element $\bconst{\F{p}}=\sspec{p}(\bconst{\Q})$ has order $3$ by \cite[Lemma 7.11]{hut:cplx13}.
\end{proof}

\begin{rem} By our main theorem,  $\ho{3}{\spl{2}{\Q}}{\zhalf{\Z}}$ has (odd) torsion of every possible size. However, the elements of order $3$ in this group all come from the obvious source: the torsion of order $3$ in 
$\spl{2}{\Z}$ and its $\gl{2}{\Q}$- conjugates in $\spl{2}{\Q}$. More precisely, a basis for the $\F{3}$-vector space $\ntor{3}{\ho{3}{\spl{2}{\Q}}{\zhalf{\Z}}}$
 is $\{\tau\}\cup \{ \tau_p\ |p\equiv -1\mod{3}\}$ where $\tau$ is the image of $1\in \Z/3=\ho{3}{\an{t}}{\Z}\to \ho{3}{\spl{2}{\Q}}{\zhalf{\Z}}$ and $\tau_p$ is the image of $1\in \Z/3=\ho{3}{\an{t}^{D_p}}{\Z}\to \ho{3}{\spl{2}{\Q}}{\zhalf{\Z}}$
with $D_p:= \mathrm{diag}(p,1)\in \gl{2}{\Q}$.
\end{rem}

\begin{rem}  Although our main results are over the coefficient ring $\zhalf{\Z}$, it is possible to say something about the $2$-torsion structure of 
$\ho{3}{\spl{2}{\Q}}{\Z}$.  Theorem \ref{thm:main} implies that $\ho{3}{\spl{2}{\Q}}{\Z}$ is a torsion group.  (This is already known from the rank calculations in \cite{borel:yang}). For any global field $F$ there is a well-defined homomorphism (induced by the maps $\sspec{v}$)
\[
\aug{F}\ho{3}{\spl{2}{F}}{\Z}\to \bigoplus_v \qpb{k(v)}
\]
where $v$ ranges over the discrete valuations. Our main theorem tells us that when $F=\Q$ the kernel and cokernel of this homomorphism are $2$-torsion groups.
In fact it can be shown that the cokernel is annihilated by $4$ (since the cokernel of each of the maps $\sspec{v}$ is annihilated by $4$).  It follows that $\ho{3}{\spl{2}{\Q}}{\Z}$ contains elements of order $2^n$ for all $n$; i.e. it also contains $2$-torsion of all possible orders. 
\end{rem}

\end{document}